\documentclass{amsart}
\usepackage{amsopn}
\usepackage{amssymb, amscd, amsmath}
\usepackage{array}
\usepackage{color}

\newcommand{\nc}{\newcommand}

\nc{\ggl}{\mathfrak{g}}
\nc{\Gm}{\Gl(\mg)}
\nc{\Gg}{\Gl(\ggl)}
\nc{\glgg}{{\mathfrak{gl}(\mathfrak{g})}}
\nc{\GHm}{\Gl^{\!H} (\mg)}
\nc{\GHmk}{\Gl^H(\mg_k)}
\nc{\gHm}{{\mathfrak{gl}^H(\mathfrak{m})}}
\nc{\OHm}{\Or^H(\mg)}
\nc{\sogHm}{ {\mathfrak{so}^H(\mathfrak{m})} }
\nc{\Vm}{V(\mg)}
\nc{\Vg}{V(\ggl)}
\nc{\Vmi}{V(\mg_\infty)}
\nc{\Om}{\Or(\mg)}
\nc{\Og}{\Or(\ggl)}
\nc{\Symm}{{\rm Sym}(\mg)}
\nc{\Symg}{{\rm Sym}(\ggl)}
\nc{\gm}{\mathfrak{gl}(\mg)}
\nc{\som}{\sog(\mg)}
\nc{\sogg}{{\mathfrak{so}(\mathfrak{g})}}
\nc{\hml}{{\mu^{\ggl}}}
\nc{\ml}{{\mu^{\ggl}_{\mg}}}
\nc{\Ol}{{\mathcal{O}_{\ggl}}}
\nc{\OlH}{{\mathcal{O}^H_{\ggl}}}
\nc{\Vn}{V(n)}
\nc{\SymV}{{\rm Sym}(V)}
\nc{\mub}{{\bar \mu}}
\nc{\mumg}{{\mu_\mg}}
\nc{\Betam}{{\Beta_\mg}}

 \nc{\Sym}{\operatorname{Sym}} 
 \nc{\vmin}{{v_{\mathsf{min}}}}

 \nc{\iph}{\la \! \la \cdot, \cdot \ra \! \ra}

\nc{\Gl}{\mathsf{GL}} \nc{\Or}{\mathsf{O}}  \nc{\SO}{\mathsf{SO}}   \nc{\Sl}{\mathsf{SL}}  
\nc{\G}{\mathsf{G}} \nc{\K}{\mathsf{K}}  \nc{\T}{\mathsf{T}} \nc{\Lsf}{\mathsf{L}}
\nc{\Qb}{\mathsf{Q}_\Beta} \nc{\Hb}{\mathsf{H}_\Beta} \nc{\Ub}{\mathsf{U}_\Beta} 
\nc{\Gb}{\mathsf{G}_\Beta} \nc{\Kb}{\mathsf{K}_\Beta}
\nc{\PPP}{\mathsf{P}}
\nc{\U}{\mathsf{U}}\nc{\Q}{\mathsf{Q}}
\nc{\LL}{\mathsf{L}}

 \nc{\ggo}{\mathfrak{g}}
 \nc{\ggob}{\overline{\mathfrak{g}}}
\nc{\lamg}{\Lambda^2\ggo^*\otimes\ggo}
\nc{\gkp}{(\ggo=\kg\oplus\pg,\ip)} \nc{\ukh}{(\ug=\kg\oplus\hg,\ip)}
\nc{\tgkp}{(\tilde{\ggo}=\kg\oplus\pg,\ip)}

\nc{\fg}{\mathfrak{f}}  \nc{\vg}{\mathfrak{v}} \nc{\wg}{\mathfrak{w}} \nc{\zg}{\mathfrak{z}} \nc{\ngo}{\mathfrak{n}} \nc{\kg}{\mathfrak{k}} \nc{\mg}{\mathfrak{m}} \nc{\bg}{\mathfrak{b}}  \nc{\sog}{\mathfrak{so}} \nc{\sug}{\mathfrak{su}} \nc{\spg}{\mathfrak{sp}} \nc{\slg}{\mathfrak{sl}} \nc{\glg}{\mathfrak{gl}} \nc{\cg}{\mathfrak{c}} \nc{\rg}{\mathfrak{r}}  \nc{\hg}{\mathfrak{h}} \nc{\tgo}{\mathfrak{t}} \nc{\ug}{\mathfrak{u}} \nc{\dg}{\mathfrak{d}} \nc{\ag}{\mathfrak{a}} \nc{\pg}{\mathfrak{p}} \nc{\sg}{\mathfrak{s}} \nc{\affg}{\mathfrak{aff}} \nc{\qg}{\mathfrak{q}}
\nc{\Xg}{\mathfrak{X}} \nc{\lgo}{\mathfrak{l}} \nc{\tg}{\mathfrak{t} }

\nc{\pca}{\mathcal{P}} \nc{\nca}{\mathcal{N}} \nc{\lca}{\mathcal{L}} \nc{\oca}{\mathcal{O}} \nc{\mca}{\mathcal{M}} \nc{\tca}{\mathcal{T}} \nc{\aca}{\mathcal{A}} \nc{\cca}{\mathcal{C}} \nc{\gca}{\mathcal{G}} \nc{\sca}{\mathcal{S}} \nc{\hca}{\mathcal{H}} \nc{\bca}{\mathcal{B}} \nc{\dca}{\mathcal{D}} \nc{\fca}{\mathcal{F}} \nc{\Qca}{\mathcal{Q}}

\nc{\dd}{{\rm d}}  \nc{\ddt}{\tfrac{{\rm d}}{{\rm d}t}}        \nc{\dds}{\tfrac{{\rm d}}{{\rm d}s}} 
\nc{\ddtbig}{\frac{{\rm d}}{{\rm d}t}}      \nc{\dpar}{\tfrac{\partial}{\partial t}}    

\nc{\im}{\mathtt{i}}    

  \nc{\Spe}{\mathrm{Sp}}          
\nc{\SU}{\mathrm{SU}}        
\nc{\Se}{\mathrm{S}}    \nc{\Cl}{\mathrm{Cl}}       \nc{\Spein}{\mathrm{Spin}}
\nc{\Pin}{\mathrm{Pin}} 

\nc{\RR}{{\mathbb R}} \nc{\HH}{{\mathbb H}} \nc{\CC}{{\mathbb C}} \nc{\ZZ}{{\mathbb Z}}
\nc{\FF}{{\mathbb F}} \nc{\NN}{{\mathbb N}} \nc{\QQ}{{\mathbb Q}} \nc{\PP}{{\mathbb P}}

\nc{\vs}{\vspace{.2cm}} \nc{\vsp}{\vspace{1cm}} 
\nc{\ip}{{\langle \,\cdot \,,\cdot \,\rangle }}
 \nc{\la}{\langle} \nc{\ra}{\rangle} \nc{\unm}{\tfrac{1}{2}}
\nc{\unc}{\tfrac{1}{4}} \nc{\und}{\tfrac{1}{16}} \nc{\no}{\vs\noindent}
\nc{\lam}{\Lambda^2(\RR^n)^*\otimes\RR^n} \nc{\tangz}{{\rm T}^{\rm Zar}}

\nc{\nor}{{\sf n}}  \nc{\mum}{/\!\!/} \nc{\kir}{/\!\!/\!\!/}
\nc{\Ri}{\tfrac{4\Ric_{\mu}}{||\mu||^2}} \nc{\ds}{\displaystyle}
\nc{\ben}{\begin{enumerate}} \nc{\een}{\end{enumerate}} \nc{\f}{\frac}
\nc{\lb}{[\cdot,\cdot]} \nc{\isn}{\tfrac{1}{||v||^2}}

\nc{\wt}{\widetilde}
\nc{\raw}{\rightarrow} \nc{\lraw}{\longrightarrow} \nc{\hqn}{\mathcal{H}_{q,n}}
\nc{\minimatrix}[4]{\left[\begin{smallmatrix} {#1} & {#2} \\ {#3} & {#4} \end{smallmatrix}\right]}
\nc{\twomatrix}[4]{\left[\begin{array}{cc} {#1} & {#2} \\ {#3} & {#4} \end{array} \right]}
\nc{\threematrix}[9]{\left[\begin{array}{ccc} {#1} & {#2} & {#3} \\ {#4} & {#5} & {#6}\\ {#7} & {#8} & {#9} \end{array} \right]}
\nc{\mut}{\tilde{\mu}} \nc{\mur}{{\mu_r}} \nc{\mutr}{{\tilde{\mu}_r}}

\nc{\alert}{\color{blue}}

\nc{\glgan}{\minimatrix{0}{0}{\star}{0}} \nc{\glgna}{\minimatrix{0}{\star}{0}{0}}  \nc{\glgnn}{\minimatrix{0}{0}{0}{\star}}  \nc{\glgaa}{\minimatrix{\star}{0}{0}{0}}
\nc{\Vaan}{{\left(\ag \wedge \ag\right)^* \otimes \ngo}} \nc{\Vann}{{\left(\ag \otimes \ngo \right)^* \otimes \ngo}} \nc{\Vnnn}{{\left(\ngo \wedge \ngo \right)^* \otimes \ngo}}

\nc{\ad}{\operatorname{ad}}  \nc{\Aut}{\operatorname{Aut}}   \nc{\Inn}{\operatorname{Inn}}   \nc{\Lie}{\operatorname{Lie}} \nc{\Ad}{\operatorname{Ad}} \nc{\Der}{\operatorname{Der}} \nc{\rad}{\operatorname{rad}} \nc{\kf}{\operatorname{B}}

\nc{\End}{\operatorname{End}} \nc{\rank}{\operatorname{rank}} \nc{\Ker}{\operatorname{Ker}} \nc{\tr}{\operatorname{tr}} \nc{\sym}{\operatorname{sym}} \nc{\diag}{\operatorname{diag}} \nc{\proy}{\operatorname{pr}} \nc{\Adj}{\operatorname{Adj}} \nc{\proj}{\operatorname{pr}} \nc{\Id}{{\operatorname{Id}}} \nc{\Span}{\operatorname{span}}

\nc{\Hess}{\operatorname{Hess}}  \nc{\dif}{\operatorname{d}} \nc{\sen}{\operatorname{sen}} \nc{\grad}{\operatorname{grad}} \nc{\Order}{\operatorname{O}} \nc{\divg}{\operatorname{div}}

\nc{\Iso}{\operatorname{Iso}} \nc{\Diff}{\operatorname{Diff}} \nc{\Rc}{\operatorname{Rc}} \nc{\Ricci}{\operatorname{Ric}}
\nc{\ric}{\operatorname{ric}} 
\nc{\Riem}{\operatorname{Rm}} \nc{\scal}{\operatorname{scal}} \nc{\scalm}{\operatorname{scal}^\star} \nc{\Riccim}{\operatorname{Ric}^{\star}} \nc{\tang}{\operatorname{T}} \nc{\vol}{\operatorname{vol}} \nc{\mcv}{\operatorname{H}} \nc{\inj}{\operatorname{inj}}
\nc{\isog}{\mathfrak{iso}}

\nc{\mm}{\operatorname{M}} \nc{\CH}{\operatorname{CH}} \nc{\Irr}{\operatorname{Irr}} \nc{\mcc}{\operatorname{mcc}} \nc{\Sb}{\mathcal{S}_\Beta} \nc{\mmm}{\operatorname{m}} 

\nc{\Beta}{{\beta}}
\nc{\Alpha}{A}
\nc{\Vr}{V_{\Beta^+}^{r}}
\nc{\Vzero}{V_{\Beta^+}^{0}}
\nc{\Vnn}{V_{\Beta^+}^{\geq 0}}
\nc{\Vnnss}{U_{\Beta^+}^{\geq 0}}
\nc{\Vzeross}{U_{\Beta^+}^{0}}
\nc{\Betap}{ {\Beta + \Vert{\Beta}\Vert^2 \Id} }
\nc{\Ap}{ {A + \Vert{A}\Vert^2 \Id} }
\nc{\zero}{ {\backslash \{0\} } }
\nc{\normmm}{{\rm F}}
\nc{\ipp}{\la\,\cdot \,,\cdot\,\ra^*_g}
\nc{\ippk}{\la\,\cdot \,,\cdot\,\ra^*_{g_k}}
\nc{\ippi}{\la\,\cdot \,,\cdot\,\ra^*_{g_\infty}}
\nc{\ipnew}{\la \la \cdot , \cdot \ra\ra}
\nc{\der}{\mathfrak{der}}
\nc{\kfm}{\widetilde{\kf}} 
\nc{\KFm}{\widetilde{\mathcal{B}}}
\nc{\KF}{\mathcal{B}}
\nc{\II}{{\mathbb I}}
\nc{\spa}{\operatorname{span}}

\theoremstyle{plain}
\newtheorem{theorem}{Theorem}[section]
\newtheorem{proposition}[theorem]{Proposition}
\newtheorem{corollary}[theorem]{Corollary}
\newtheorem{lemma}[theorem]{Lemma}

\theoremstyle{definition}
\newtheorem{definition}[theorem]{Definition}
\newtheorem{notation}[theorem]{Notation}

\theoremstyle{remark}
\newtheorem{remark}[theorem]{Remark}

\title{Real Geometric Invariant Theory}

\author{Christoph B\"ohm}
\address{University of M\"unster, Einsteinstra{\ss}e 62, 48149 M\"unster, Germany}
\email{cboehm@math.uni-muenster.de}

\author{Ramiro A.~ Lafuente}
\address{University of M\"unster, Einsteinstra{\ss}e 62, 48149 M\"unster, Germany}
\email{lafuente@uni-muenster.de}

\thanks{The second author was supported by the Alexander von Humboldt Foundation.}

\begin{document}

\begin{abstract}
For linear actions of real reductive Lie groups we prove 
the Kempf-Ness Theorem about closed orbits and the Kirwan-Ness Stratification Theorem of the null cone.
Since our completely self-contained proof focuses strongly on geometric and analytic methods,
essentially avoiding any deep algebraic result, it applies also to non-rational linear actions.
\end{abstract}

\maketitle

\tableofcontents

\section{Introduction}

The Kempf-Ness Theorem provides a beautiful and simple geometric criterion for the closedness
of  orbits of a holomorphic representation of a complex reductive Lie group \cite{KN79}. It implies that a non-closed orbit with positive distance to the origin contains a non-trivial closed orbit in its closure. 
For orbits in the \emph{null cone}, being the union of all the orbits containing the origin
in its closure, this is no longer true. The Kirwan-Ness Theorem describes a Morse-type stratification of the null cone into finitely many invariant submanifolds, 
with respect to a natural energy functional associated to the moment map of the action  \cite{Krw1}, \cite{Ness}. 

In 1990, Richardson and Slodowy showed that the Kempf-Ness Theorem extends to the case of real reductive Lie groups acting linearly on Euclidean vector spaces \cite{RS90}; see also  \cite{Mar2001}, \cite{HSchw07}, \cite{EbJbl09}, \cite{BZ16}. Later on, the Stratification Theorem was extended to the case of real reductive Lie groups by Lauret \cite{standard},
and by Heinzner, Schwarz and St\"otzel to the more general setting of actions on complex spaces \cite{HSS}.
The proof of most of these results rely on those of the complex case,
thus making use of deep results from the theory of algebraic groups (e.g.~ \cite{Most55}, \cite{BHC62}, \cite{BT65}, \cite{Birkes}) or of complex spaces.

The aim of this article is to provide completely self-contained 
proofs of the Kempf-Ness Theorem and the Kirwan-Ness Stratification Theorem for linear actions of real reductive Lie groups.
This is achieved by adapting to our context some of the existing proofs, together with a detailed understanding of the particular case of abelian groups. 



Our setup is as follows:
let $\rho : \G \to \Gl(V)$ be a faithful representation of the real Lie group $\G$ on a real,
finite-dimensional   vector space $V$, with $\rho(\G)$ closed 
in $\Gl(V)$. To simplify notation, in what follows we will suppress $\rho$. We say that a closed subgroup $\G\subset \Gl(V)$
is a \emph{real reductive Lie group} if
there exists a scalar product $\ip$ on $V$ such that 
\begin{equation}\label{eqn_Cartandec}
	\G = \K \cdot \exp(\pg), 
\end{equation}
where $\K:=\G \cap \Or(V,\ip)$, $\pg:= \ggo \cap \Sym(V,\ip)$, $\ggo$ denotes the Lie algebra of $\G$, and $\exp:\glg(V) \to \Gl(V)$ the Lie exponential map. Here $\Or(V,\ip)$ denotes the group of orthogonal linear maps in $\Gl(V)$ and $\Sym(V,\ip)$ the set of symmetric endomorphisms of $V$.  
The maximal compact subgroup of $\G$ is
$\K$, and at Lie algebra level \eqref{eqn_Cartandec} yields a Cartan decomposition $\ggo = \kg \oplus \pg$, that is $[\kg,\pg]\subset \pg$ and $[\pg,\pg]\subset \kg$.
Let us mention
that there are several non-equivalent definitions of real reductive Lie group
in the literature. We refer the reader to Appendix \ref{app_reductive} for a comparison between these definitions and ours.

The group $\Gl(V)$ itself is real reductive, and so is any
faithful, finite-dimensional representation of a real semisimple Lie group $\G$ with finitely many components \cite{Most55}. The same is true for $\Gl_n(\RR)$, provided the center acts by semisimple endomorphisms. For $\G \subset \Gl(V)$ connected, \eqref{eqn_Cartandec} is equivalent to saying that its Lie algebra is closed under transpose, see \cite[Prop.~7.14]{Knapp2e}.
 In contrast to these examples, notice that the action of a nilpotent, non-abelian Lie group is never real reductive. 

We turn now to the Kempf-Ness Theorem.
A vector $\bar v \in \G\cdot v\subset V$ is called a \emph{minimal vector},
if it minimizes the distance to $0\in V$ within the orbit $\G \cdot v$.
If we denote by $\mca \subset V$ the set of all minimal vectors, it is clear that
any closed orbit must intersect $\mca$. Conversely, we have the following 

\begin{theorem}\label{thm_realGIT}
For a real reductive Lie group $\G$ acting linearly on $(V,\ip)$ the following holds:
\begin{itemize}
	\item[(i)] Any orbit $\G\cdot v$ containing a minimal vector $\bar v$ is closed, and moreover we have that $\G\cdot \bar v \cap \mca = \K \cdot \bar v$.
	\item[(ii)] If the orbit $\G\cdot v$ is not closed, then there exists $\alpha \in \pg$ such that 
	the limit $w = \lim_{t\to\infty}\exp(t \alpha) \cdot v$ exists, and the orbit $\G\cdot w$ is closed.
	\item[(iii)] The closure of any orbit contains exactly one closed orbit.
	\item[(iv)] The null cone $\nca = \big\{ v\in V : 0\in \overline{\G\cdot v}  \big\}$ is a closed subset of $V$.
\end{itemize}
\end{theorem}

Part (iii) of the above theorem was first proved in \cite{Lu75}, and
parts (i) and (ii) in \cite{RS90}, under the assumption that the action 
of $\G$ on $V$ is \emph{rational} in the sense of algebraic geometry. One of its main implication is the fact that the set of closed orbits provides a \emph{good quotient} for the $\G$-action, with much better properties than the potentially non-Hausdorff orbit space.
In the complex case, part (i) of the above theorem is known as the Kempf-Ness Theorem \cite{KN79}, 
part (ii) is related to the Hilbert-Mumford criterion for stability \cite{GIT1994}, and part (iii) appears in \cite{Lu73}.
Let us mention that the open set of \emph{semistable} vectors $V\backslash \nca$ is either empty or dense, by \mbox{\cite[Appendix A]{HS10}}. However, at the moment our methods do not allow us to prove this fact in an elementary way.


Using  the decomposition $\G = \K \T \K$, see Appendix \ref{app_reductive},
the proof of Theorem \ref{thm_realGIT} can be reduced to the abelian case. Here $\T$ is a maximal subgroup of $\G$ contained in $\exp(\pg)$, necessarily abelian and non-compact.
The abelian case relies 
on two crucial facts:  the convexity of the distance function to the origin along one-parameter subgroups and the separation of any two {}
closed $\T$-invariant sets by continuous $\T$-invariant functions.

We turn now to the Stratification Theorem. Endow $\ggo$ with an $\Ad(\K)$-invariant scalar product, also denoted by $\ip$, such that $\langle \kg,\pg\rangle=0$ 
and
\begin{equation}\label{eqn_assumipggo}
    \ad(\kg) \subset \sog(\ggo,\ip), \qquad \ad(\pg) \subset \sym(\ggo,\ip).
\end{equation}
For instance,  one possible choice for $\ip$ is the restriction of the usual scalar product on $\glg(V)$ induced by that on $V$.

\begin{definition}[Moment map]\label{def_mm}
The map  $\mmm : V \backslash \{0 \} \to \pg$ defined implicitly by
\begin{equation}\label{eqn_defmm}{}
    \la \mmm(v), \Alpha \ra 
    = \tfrac1{\Vert v\Vert^2} \cdot \la \Alpha \cdot v, v\ra \,,
\end{equation} 
for all $\Alpha\in \pg$, $v \in V \backslash \{ 0\}$,
is called the \emph{moment map} associated to the action of $\G$ on $V$.
The corresponding \emph{energy map} is given by
\[
    \normmm : V\backslash \{ 0\} \to \RR\,\,;\,\,\, v \mapsto \Vert {\mmm(v)} \Vert^2\,.
\] 
\end{definition}

This scale-invariant moment map describes the infinitesimal change of the norm in $V$ under the group action: if $(\exp(tA))_{t \in \RR}$ is the one-parameter subgroup of $\G$ associated to $A \in \pg$, then the
corresponding smooth action field on $V$ is given by $X_A(v): =A\cdot v$. 
It follows that minimal vectors are zeroes of $\mmm$. In fact, it can be proved that $\mca = \mmm^{-1}(0)$, see Lemma \ref{lem_convex}. The $\K$-invariance of the involved scalar products implies that the moment map $\mmm$ is $\K$-equivariant, if we consider on $\pg$ the adjoint action of  $\K$. That is, we have $\mmm(k \cdot v) = k \, \mmm(v) \, k^{-1}$,
for all $k\in \K$. 
The name \emph{moment map} comes from symplectic geometry, see  Section \ref{sec_complex}.

The following real version of
the Kirwan-Ness Stratification Theorem \cite{Krw1}, \cite{Ness} 
 is due to \cite{HSS} and \cite{standard}

\begin{theorem}\label{thm_stratif}
There exists a finite subset $\bca \subset \pg$ 
and a collection of smooth, $\G$-invariant submanifolds 
$\{ \sca_\Beta \}_{\Beta \in \bca}$ of $V$, with the following properties:
\begin{itemize}
    \item[(i)]  
     We have $V\backslash \{ 0\} = \bigcup_{\Beta\in \bca} \sca_\Beta$
     and $\sca_\Beta \cap \sca_{\Beta'}=\emptyset$ for $\Beta \neq \Beta'$.
    \item[(ii)] We have 
     $\overline{\sca_\Beta} \, \backslash \, \sca_\Beta \subset \bigcup_{\Beta'\in \bca, \Vert \Beta'\Vert > \Vert \Beta \Vert} \sca_{\Beta'}$ (the closure taken in $V\backslash \{ 0\}$).
    \item[(iii)] A vector $v$ is contained in
       $\sca_\Beta$ if and only if the negative gradient 
             flow of $\normmm$ starting at $v$ 
             converges to a critical point $v_C$ of $\normmm$
             with $\mmm(v_C) \in \K \cdot \Beta$. 
\end{itemize} 
\end{theorem}

The submanifolds $\sca_\Beta$ are called \emph{strata}. 
The set of semistable vectors $V \backslash \nca$ is  nothing but the stratum 
$\sca_0$ (Corollary \ref{cor_semistablesca0}).
The $\G$-invariance (and scale-invariance) of the strata is justified by the formula for the gradient of $\normmm$ given in Lemma \ref{lem_gradient}.

It is worthwhile to mention that for some 
applications in non-K\"ahlerian Riemannian geometry
it is interesting to consider representations where $\nca = V$. The reason for this is
that from the proof of the Stratification Theorem one can deduce estimates for the 
associated moment map, which are trivial on $\sca_0$ 
but highly non-trivial on $\nca$: see \cite{standard}, \cite{BL17}
 and Lemma \ref{lem_mainestimate}.


Concerning the proof of Theorem \ref{thm_stratif} we recall  that
the energy map $\normmm$ is in general not a Morse-Bott function. Nevertheless, it has the following remarkable property: 
the image of its critical points under the moment map consists of finitely many $\K$-orbits 
$\K\cdot \Beta_1,\ldots,K\cdot \Beta_N$. As a consequence,
 $\bca=\{\Beta_1,\ldots,\Beta_N\}$ is a finite set.
For $\Beta \in \bca$ one denotes by $\cca_\Beta$ the set of critical points of
$\normmm$ with $\mmm(\cca_\Beta) \subset \K\cdot \Beta$ and by
$\sca_\Beta \subset V$ the unstable manifold of $\cca_\Beta$
with respect to the negative (analytic) gradient flow of $\normmm$. One then needs to prove that $\sca_\Beta$ is $\G$-invariant. The proof of the Stratification theorem
does in fact go the other way around: one first defines certain sets as \emph{candidates} for being strata, and then proves that they are  invariant under the negative gradient flow of $\normmm$. See Section \ref{sec_strat} for further details.


The article is organized as follows. In Section \ref{sec_examples} we discuss three explicit examples that illustrate the basic concepts. In Section \ref{sec_complex} we explain how our setting is related to the notion of moment map in symplectic geometry. 
 In Section \ref{sec_torusactions},
we assume that $\G=\T$ is abelian, and prove that any
two disjoint closed  $\T$-orbits can be separated by a single $\T$-invariant continuous function. More generally, in Section \ref{sec:sepclosedT} we show that
two disjoint closed  $\T$-invariant sets can be separated by a $\T$-invariant continuous function.
In Section \ref{sec_generalactions} we generalize this
 to real reductive Lie groups and complete the proof of Theorem \ref{thm_realGIT}.
The Stratification Theorem \ref{thm_stratif} is proved in Sections \ref{sec_strat} and \ref{sec_critical}. In Section \ref{sec_applications} we mention some immediate applications of the stratification theorem.
Finally, the two appendices contain some well-known Lie-theoretic properties of real reductive Lie groups and its subgroups, which we prove based solely on our assumption \eqref{eqn_Cartandec}. 

\vs \noindent {\it Acknowledgements.} It is our pleasure to thank Ricardo Mendes and Marco Radeschi for fruitful discussions, and Michael Jablonski, Martin Kerin and Jorge Lauret for their helpful comments.

\section{Examples}\label{sec_examples}

In order to illustrate the content of Theorems \ref{thm_realGIT} and \ref{thm_stratif} we describe in this section three concrete examples.

First of all we consider one of the simplest examples of an action satisfying \eqref{eqn_Cartandec}, namely $\G=\RR_{>0}$ acting on $V=\RR^2$ via
 $  \lambda \cdot (x,y) := (\lambda x, \lambda^{-1} y)$,
 $ \lambda \in \RR_{>0}$ and $x,y\in \RR$.
The null cone is the union of the origin and $4$ non-closed orbits, two for each axis. All the orbits corresponding to semistable vectors $(x,y)\in \RR^2$, $x,y\neq 0$, are closed. Finally, the set of minimal vectors is $\mca = \{ (x,y)\in \RR^2 : \vert x \vert = \vert y \vert \}$, and the moment map is $\mmm(x,y) = (x^2 - y^2)/(x^2+y^2)$. Notice that by changing the action, say by replacing $\lambda x$ by $\lambda^\pi x$ on the right-hand side, one obtains a representation whose complexification is not rational in the sense of algebraic geometry.

Next, consider the less trivial example of $\G=\Sl_n(\RR)$ acting by conjugation on the space $V={\rm Mat}(n,\RR)$ of $(n\times n)$-matrices with real entries:
\[
    h \cdot A := h \, A \, h^{-1}\,,
\] 
where  $h\in \Sl_n(\RR)$ and $A\in {\rm Mat}(n,\RR)$.
We endow $\slg_n(\RR)$ and ${\rm Mat}(n,\RR)$ with the usual scalar product induced from that of $\RR^n$, $\la A, B\ra = \tr A B^t$. In this case, for any $A\in {\rm Mat}(n,\RR)$ the orbit $\Sl_n(\RR) \cdot A$ is closed if and only if $A$ is \emph{semisimple} (i.e.~ diagonalizable over $\CC$). More generally, let $A = S+N$ denote the Jordan decomposition,
that is $S,N \in {\rm Mat}(n,\RR)$,  $S$ semisimple, $N$ nilpotent and 
\mbox{$[S,N] = 0$}.
 Then $A$ is semistable if and only if $S\neq 0$. Thus the null cone consists of the set of
 nilpotent matrices. If $\pg_0 \subset \slg_n(\RR)$ denotes the subset of traceless symmetric matrices, the moment map is given by
\[
    \mmm : {\rm Mat}(n,\RR) \backslash \{ 0\} \to \pg_0, \qquad \mmm(A) = \tfrac1{\Vert A\Vert^2}\cdot [A, A^t].
\] 
The minimal vectors are  the normal matrices. Moreover, the strata $\sca_{\Beta}$, $\Beta \neq 0$,
 are in one-to-one correspondence with Jordan canonical forms for a nilpotent matrix 
 $N$. They are  parameterized by ordered partitions $n = n_1 + \cdots + n_r$, $n_1 \leq \cdots \leq n_r$. For explicit computations of the corresponding stratum labels and critical points of $\normmm$ for this example we refer the reader to \cite[$\S$4]{finding}.

Our last example, which was in fact our main motivation for writing this article, is as follows: let $V_n := \Lambda^2(\RR^n)^* \otimes \RR^n$ denote the vector space of skew-symmetric, bilinear maps $\mu : \RR^n \times \RR^n \to \RR^n$, and consider the \emph{change of basis} action of $\G=\Gl_n(\RR)$ on $V_n$, given by
\[
    h\cdot \mu (\cdot, \cdot) := h \mu(h^{-1} \cdot , h^{-1} \cdot)\,,
\] 
where $ h\in \Gl_n(\RR)$ and  $\mu\in V_n$.
Let $\{e_i\}_{i=1}^n$ denote the canonical basis of $\RR^n$ and $\{e_i^*\}_{i=1}^n$ its dual. We endow $\glg_n(\RR) \simeq (\RR^n)^* \otimes \RR^n$ and $V_n$ with the scalar products making the respective bases $\{ e_i^* \otimes e_j \}_{i,j}$ and $\{ e_i^* \wedge e_j^* \otimes e_k \}_{i<j; k}$ orthonormal. Notice that this scalar product
is \emph{not} the one induced from $\glg(V_n)$ via the above action.  

Observe now that $V_n$ contains as an algebraic subset the so called \emph{variety of Lie algebras}
\[
    \lca_n := \{ \mu \in V_n : \mu \hbox{ satisfies the Jacobi identity}\}.
\]
The set $\lca_n$ is $\Gl_n(\RR)$-invariant, and an orbit $\Gl_n(\RR) \cdot \mu\subset \lca_n$ consists precisely of those Lie brackets $\tilde \mu$ which are isomorphic to $\mu$. The null cone in this case is everything, since $\Id \in \glg_n(\RR)$ acts as $-\Id_{V_n}$, thus any orbit contains $0$ in its closure. However, by restricting to the action of $\Sl_n(\RR)$ there exist closed orbits, and those in $\lca_n$ correspond precisely to the semisimple Lie algebras. 

An additional feature is the fact that the variety $\lca_n$ can be thought of as a parameterization of the space of left-invariant Riemannian metrics on $n$-dimensional Lie groups \cite{Lau2003}. The moment map $\mmm : V_n\backslash\{0\} \to \pg,$ where $\pg \subset \glg_n(\RR)$ denotes the subset of symmetric matrices, can be computed explicitly, and it appears naturally in the formula for the Ricci curvature of the corresponding Riemannian metric. Furthermore, within the subset of nilpotent Lie algebras, the critical points for $\normmm$ correspond to \emph{Ricci soliton} nilmanifolds \cite{soliton}.


\section{Comparison with complex and symplectic case}\label{sec_complex}

In this section we connect our setting with the complex setting, and explain how this relates with the notion of moment map from symplectic geometry. 

We first recall that, despite the fact that the definition of real reductive groups used by Richardson and Slodowy relies on that of a complex reductive Lie group, it follows from \cite[2.2]{RS90} that they all satisfy our assumption \eqref{eqn_Cartandec}.



 Let now $\G\subset \Gl(V)$ be a closed subgroup satisfying \eqref{eqn_Cartandec}, and let $V^\CC = V \otimes_\RR \CC$ be the complexified vector space. For simplicity let us assume for the rest of this section that $\G$ is connected. After considering the natural inclusion $\Gl(V) \subset \Gl(V^\CC)$, the complexification $\ggo^\CC$ of the Lie algebra $\ggo$  of $\G$ can be viewed as a subalgebra of $\glg(V^\CC)$. Let $\G^\CC$ be the connected Lie subgroup of $\Gl(V^\CC)$ with Lie algebra $\ggo^\CC$. We call $\G^\CC$ the \emph{complexification} of the Lie group $\G$: it is a complex Lie group containing $\G$ as a closed subgroup, and whose Lie algebra is the complexification of that of $\G$. We will assume for simplicity that $\G^\CC$ is a closed subgroup of $\Gl(V^\CC)$, although this is not necessarily always the case.

The inner product $\ip$ on $V$ induces as usual a Hermitian inner product $\iph$ 
 on $V^\CC$, which allows us to identify $V^\CC \simeq \CC^{N+1}$ as complex vector spaces, so that $\iph$ becomes the canonical Hermitian inner product on $\CC^{N+1}$. Since $\G^\CC$ acts linearly on $\CC^{N+1}$, it also acts on the corresponding projective space $\CC P^N$. Moreover, 
 the condition \eqref{eqn_Cartandec} implies that the maximal compact subgroup $\U$ of $\G^\CC$ (whose Lie algebra is given by $\ug = \kg \oplus i \pg$) acts by unitary transformations on $(V^\CC, \iph)$, and in particular its corresponding action on $\CC P^N$ preserves the Fubini-Study metric and its associated $2$-form $\omega_{FS}$. For example, when $\G = \Gl(V)$ we have $\U = \U(N+1)$, $N+1 = \dim_\CC V^\CC$. 
 It is well-known (see e.g.~ \cite[Lemma 2.5]{Krw1}) that for this symplectic action of the compact Lie group $\U$ on $\CC P^N$ there exists a moment map in the sense of symplectic geometry,
\[
	\mmm_\U : \CC P ^N \to \ug^*, \qquad \mmm_\U(x) (A) = 
	\tfrac{v^t A v}{2 \pi \Vert v \Vert^2},
\]
where $v\in \CC^{N+1}$ is any vector over $x \in \CC P^N$ and $A\in \ug \subset \ug(N+1)$.
The inclusion $i \pg \subset \ug$ induces a restriction map $\operatorname{r}_{\pg} : \ug^* \to (i \pg)^* \simeq \pg$, the last identification being made using the scalar product on $\pg \subset \ggo$. Up to a constant scalar multiple, the moment map for the action of $\G$ on $V$ (Definition \ref{def_mm}) satisfies 
\[
	\mmm = \operatorname{r}_{\pg} \circ \mmm_\U \, \circ \, \pi \, \big|_{V \backslash \{ 0\} },
\]
 where $\pi : V^\CC \backslash \{ 0\} \simeq \CC^{N+1} \backslash \{0 \} \to \CC P^N$ is the usual projection.

\section{The abelian case}\label{sec_torusactions}

In this section we assume that $\K = \{ e\} $ and that $\G = \T = \exp(\tg)$ is an abelian group of positive definite matrices, $\tg \subset \Sym(V,\ip)$. 
Since commuting symmetric matrices can be diagonalized simultaneously, 
there exists an orthonormal basis $\{e_1, \ldots, e_N\}$  for $V$ which diagonalizes the action of $\T$. Let \mbox{$\alpha_1, \ldots, \alpha_N \in \tg$} 
be the corresponding ``weights'', that is, the action of $\T$ on $V$ is given by 
\begin{equation}\label{eqn_Taction}
	\exp(\lambda) \cdot v = \big(e^{\la \lambda, \alpha_1\ra} v_1, \ldots, e^{\la \lambda, \alpha_N\ra} v_N\big)\,,
\end{equation}
where $\lambda \in \tg$ and  $v = (v_1,  \ldots,  v_N) \in V$, the coordinates being with respect to the chosen basis. The scalar product $\ip$ on $\tg$ is simply given by restricting the one on $\ggo$ (see the paragraph before Definition \ref{def_mm}).

 For any subset $I \subset \II_N := \{1, \ldots, N\}$ we set $\tg_I := \spa_\RR \{ \alpha_i : i\in I\} \subset \tg$
with the convention that $\tg_\emptyset=\{0\}$. Moreover we 
define the vector subspace
$V_I := \left\{ v \in V : v_i =0 \, \, \mbox{for all } i\notin I \right\}$ of $V$
and the open subsets
\begin{align*}
	U_I := \left\{ v  \in V_I : v_i  \neq 0 \, \, \mbox{for all } i\in I \right\} \,\,,
	\quad 
    U_I^+ := \left\{ v  \in U_I : v_i  > 0 \, \, \mbox{for all } i\in I \right\} 
\end{align*}
of $V_I$.
Clearly,  $U_I$ is a dense subset of $V_I$, disconnected if $I \neq \emptyset$, and $U_I^+$ is one of its connected components. 
Notice that $V = \cup_{I\subset \II_N} U_I$ as a  disjoint union.

\begin{lemma}[Hilbert-Mumford criterion for abelian groups]\label{lem_oneparamT}
Let $v \in V$ and suppose that $\T\cdot v$ is a non-closed orbit.
Then,
 for any $\bar v \in \overline{\T\cdot v} \, \backslash \T\cdot v$ there exists 
$\alpha \in \tg$ and $g \in \T$ such that $\lim_{t\to \infty} \exp(t \alpha) \cdot v = g \cdot \bar v$. 
\end{lemma}

\begin{proof}
Let $I, J\subset \II_N$ be such that $v\in U_I$, $\bar v\in U_J$. The assumptions imply that $J \subset I$. But if $J =I$ then  $\bar v \in \T \cdot v$. Thus $J \subsetneq  I$, from which $J^C:=I\backslash J \neq \emptyset$.

Let $\left(\lambda^{(k)}\right) \subset \tg$ be a sequence with $\lim_{k \to \infty}\exp(\lambda^{(k)}) \cdot v =\bar v$. From \eqref{eqn_Taction} we deduce that
for all $j \in J$ and all $i \in J^C$ it holds that
\begin{eqnarray}\label{eqn_seqlambdabdd1}
	 \lim_{k \to \infty}\la \lambda^{(k)}, \alpha_j\ra 
	  = 
	  \lambda^\infty_j \in \RR  \quad\textrm{ and }\quad
	\lim_{k \to \infty}\la \lambda^{(k)}, \alpha_i\ra 
	&=& 
	-\infty\,.
\end{eqnarray}
In particular, the projection of $\lambda^{(k)}$ onto $\tg_J$ converges to some $\lambda^\infty \in \tg_J$ as $k\to \infty$.
We decompose $\tg_I = \tg_{J} \oplus \tg_{J}^\perp$  orthogonally 
 and for each $i\in J^C$ denote  by 
$\alpha_i^\perp\neq 0$ the orthogonal projection of $\alpha_i$ onto $\tg_J^\perp$. We claim, that
$0$ is not contained in the convex hull $\cca = CH\{ \alpha_i^\perp : i\in J^C\}$.
 Indeed, if $0 \in \cca$, then for some $c_i >0$ we would have that $\gamma := \sum_{i\in J^C} c_i \, \alpha_i \in \tg_J$, and by \eqref{eqn_seqlambdabdd1}
 the sequence $(\la \lambda^{(k)}, \gamma\ra)$ would be   bounded
 and unbounded simultaneously, a contradiction.

 Thus, let $\beta \in \cca$ be the element of minimal (positive) norm in $\cca$. We have that $\beta \perp \alpha_j$ for all $j\in J$, and convexity implies that $-\la \beta, \alpha_i\ra = -\la \beta, \alpha_i^\perp\ra \leq -\Vert \beta \Vert^2 < 0$ for all $i\in J^C$. 
 We obtain  $g \cdot \bar v = \lim_{t\to \infty} \exp(- t \beta) \cdot v$, where
  $g = \exp(-\lambda^\infty)$.
\end{proof}

\begin{corollary}\label{cor_oneclosedorbit}
Any $\T$-orbit has a closed $\T$-orbit in its closure.
\end{corollary}

\begin{proof}
By Lemma \ref{lem_oneparamT}, an orbit in the closure has strictly smaller dimension because the direction $\alpha$ defining the one-parameter subgroup is a new element in the isotropy subalgebra, easily seen by applying $\exp(s\alpha)$ to
$g \cdot \bar v=\lim_{t \to \infty}\exp(t\alpha)\cdot v$.
The claim follows now by picking an orbit of minimal dimension.
\end{proof}

For any  subset $I \subset \II_N$ we denote now by
\[
    \Delta_I := \Big\{ \sum_{i\in I} c_i \alpha_i : c_i \geq 0, \sum c_i = 1 \Big\} \subset \tg_I
\] 
the convex hull of the set of weights $\{ \alpha_i\}_{i\in I}$, with
the convention that $\Delta_\emptyset=\{0\}$. 
Notice that $\dim \Delta_I=\dim \tg_I$ and that $\Delta_I \neq \tg_I$,
provided that $\tg_I \neq \{0\}$.
Since the relative interior of a point
is that point, we obtain the following characterization of closed $\T$-orbits:

\begin{lemma}
For $v\in V$ let $I \subset \II_N$ with $v\in U_I$. Then, the orbit $\T \cdot v$ is closed if and only if $0\in \tg$ is in the relative interior $ (\Delta_I)^o$ of $\Delta_I$.
\end{lemma}

\begin{proof}
If $0$ is not in the interior of $\Delta_I$ then $\Delta_I \neq \{0\}$ and 
there exists a hyperplane $H \subset \tg_I$,
such that $H$ does not contain $\Delta_I$ and such that $\Delta_I$ does not intersect one of the two open half-spaces defined by $H$. Thus for one of the two unit normal vectors  $\beta \in \tg_I$ to $H$ we have that $\la \beta,  \alpha_i \ra \geq 0$ for all $i\in I$, and the inequality is strict for some $i_0\in I$. Hence $\bar v = \lim_{t\to\infty} \exp(-t \beta) \cdot v$ exists by \eqref{eqn_Taction}, and we have that $\bar v \notin \T\cdot v$ since $\T\cdot v \subset U_I$ and $\bar v \notin U_I$ because $(\bar v)_{i_0} = 0$.

Conversely, if the orbit is not closed then the proof of the Lemma \ref{lem_oneparamT} implies the existence of a $\beta \in \tg$ with the same properties, from which it follows that $0$ is not in the interior of $\Delta_I$.
\end{proof}

Recall that $0 \in V_\emptyset$ and that then $\Delta_\emptyset =\{0\}$.
If $v \neq 0$, then there exists a non-empty $I \subset \II_N$ with
$v \in U_I$. The above motivates now the following 

\begin{definition}
We call a non-empty subset $I\subset \II_N$ \emph{admissible}, if $0 \in (\Delta_I)^o$.
\end{definition}

Notice that if $0 \in (\Delta_I)^o$, then
there exist positive coefficients $c_i>0$, $i\in I$, such that $0 = \sum_{i\in I} c_i \, \alpha_i$.  

In what follows we aim to show that given two different closed orbits, there exists a continuous $\T$-invariant function which separates them, and that moreover we can pick finitiely many such functions to separate any two closed orbits. 
To this end, we consider the linear map
\[
    \phi : \tg \to V\,\,;\,\,\, \lambda
     \mapsto \big( \la \lambda, \alpha_1 \ra, \ldots,  \la \lambda, \alpha_N\ra \big)\,.
\]

\begin{lemma}
A subset $I\subset \II_N$ is admissible if and only if $U_I^+ \cap \phi(\tg)^\perp \neq \emptyset$.
\end{lemma}
\begin{proof}
Notice that $0 \in \big( \Delta_I\big)^o$ is equivalent to the existence of $w\in U_I^+$ such that $ 0= \sum_{i\in I} w_i \alpha_i$. But this is equivalent to
$0  = \sum_{i\in I} \big\la \lambda, w_i \alpha_i  \big\ra = \la \phi(\lambda), w\ra$
for all $\lambda \in \tg$
from which the lemma follows.
\end{proof}

Let us now fix an admissible subset $I\subset \II_N$, and let 
\[
 \bca_I := \{ w^{(1)}, \ldots, w^{(r)} \} \subset U_I^+ \cap \phi(\tg)^\perp
\] 
be a basis for $V_I \cap \phi(\tg)^\perp$ consisting of elements with positive entries for $i\in I$. Moreover, we  scale the basis elements so that for each $w\in \bca_I$ the sum of its entries is $1$. For each $w\in \bca_I$ consider the real-valued function
\[
 f^+_w : V \to \RR\,\,;\,\,\,
    v \mapsto \begin{cases}
         \prod_{i=1}^N v_i^{w_i}, \qquad & \mbox{if } v_i > 0 \hbox { for all } i\in I;\\
        0, \qquad &\mbox{otherwise}.
    \end{cases}
\]
Notice, that if $\proj_I:V\to V_I $ denotes the orthogonal projection onto $V_I$, then
\[
   f^+_w(v)=f^+_w(\proj_I(v)) \quad \textrm{ and }\quad 
\operatorname{supp}(f^+_w)=\overline{(\proj_I)^{-1}(U_I^+)}\,.
\]

\begin{lemma}\label{lem_f_w}
The function $f^+_w$ is continuous, $\T$-invariant, and $f^+_w(c \cdot v) = c \cdot f^+_w(v)$ for all $c>0$ and $v\in V$.
\end{lemma}
\begin{proof}
Continuity is clear, since  $w_i \geq 0$ for all $i \in \II_N$. To prove $\T$-invariance, first observe that $V_I\backslash U_I^+$ is a $\T$-invariant set. On the other hand, for 
$\tilde v=\proj_I(v)\in U_I^+$ we compute directly using \eqref{eqn_Taction} and the fact that $w\in \phi(\tg)^\perp$:
\[
    f_w^+( \exp(\lambda) \cdot v) = \prod_{i=1}^N (e^{\la \lambda, \alpha_i\ra} v_i)^{w_i} = e^{\la \lambda, \sum w_i\alpha_i\ra} \, f^+_w (v) = e^{\la \phi(\lambda), w\ra} \, f_w^+ (v) = f_w^+(v)\,.
\] 
This shows the claim.
\end{proof}

\begin{lemma}\label{lem_sepUIplus}
If $\oca_1 \neq \oca_2$ are two closed $\T$-orbits in $U_I^+$ then there exists $w\in \bca_I$ such that $f_w^+(\oca_1) \neq f_w^+(\oca_2)$.
\end{lemma}
\begin{proof}
Assume that this is not the case. Let $v \in \oca_1$, $\bar v\in \oca_2$, and consider the map $\log_I : U_I^+ \to V_I$, assigning to each vector $v\in U_I^+$ the vector $\log_I (v) \in V_I$ whose $i$-th entry is $\log(v_i)$, for all $i\in I$. For each $w \in \bca_I$ we have that 
\[
    \la w, \log_I (v)\ra = \log f_w^+(v) =\log f_w^+(\bar v) = \la w, \log_I (\bar v)\ra\,.
\]
Hence $\log_I (v) - \log_I (\bar v) \perp V_I \cap \phi(\tg)^\perp$. In other words, $\log_I (v) - \log_I (\bar v) \in V_I \cap \phi(\tg)$, from which it immediately follows that $v\in \T\cdot \bar v$. Contradiction.
\end{proof}

In order to separate orbits that lie in different connected components of $U_I$, we argue as follows: for each choice of signs $\sigma \in \{ \pm 1 \}^N$, let $T_\sigma : V \to V$ be the $\T$-equivariant linear map that changes the sign of each coordinate according to $\sigma$. For any connected component $U_I^c$ of $U_I$ there exists $\sigma \in \{ \pm 1\}^N$ such that $T_\sigma (U_I^c) = U_I^+$. We then define the functions $f_w^\sigma$, $w\in \bca_I$, by
$f_w^\sigma  = f_w^+ \circ T_\sigma$.
Clearly, they satisfy Lemma \ref{lem_f_w}, and they separate orbits in the corresponding connected component $U_I^c$ of $U_I$. We
consider now the finite set of continuous, $\T$-invariant, real-valued functions on $V$:
\[
    \fca := \big\{ f_w^\sigma  \, \, : \, \,  w\in \bca_I, \sigma \in \{ \pm 1\}^N, I\subset \II_N \hbox{ admissible} \big\}.
\]  
Notice that  $\operatorname{supp} (f_w^\sigma) \cap V_I = \overline {U_I^c}$.

\begin{proposition}[Separation of closed orbits]\label{prop_sepclosedorbits}
Let $L = \vert \fca \vert \in \NN$. Then there exists a continuous, $\T$-invariant map 
$\Phi : V \to \RR^L$, 
such that $\Phi(\oca_1) \neq \Phi(\oca_2)$ for any two closed orbits $\oca_1 \neq\oca_2$.
\end{proposition}

\begin{proof}
The coordinate functions of the map $\Phi$ are of course just functions
$f_w^\sigma$ in $\fca$.

First assume that $\oca_1, \oca_2 \subset U_I$. If they belong to the same connected component of $U_I$, which without loss of generality we may assume to be $U_I^+$, then they are separated by Lemma \ref{lem_sepUIplus}. On the other hand, if this is not the case then the existence of a separating function follows 
immediately from $\operatorname{supp} (f_w^\sigma) \cap V_I = \overline {U_I^c}$. 

We are left with the case $\oca_1 \subset U_I$, $\oca_2 \subset U_J$ with $I\neq J$.
Suppose that $j\in J \backslash I$. Then, there exists $f_w^\sigma \in \fca$ with
 $w\in \bca_J$ such that $f_w^\sigma(\oca_2) > 0$ and $f_w^\sigma(\oca_1) = 0$.
\end{proof}

\section{Separation of closed $\T$-invariant sets}\label{sec:sepclosedT}

The linear action of the real reductive group $\G\subset \Gl(V)$ on $(V,\ip)$
provides us with a smooth {\it action field} for any $A \in \ggo$:
\[
   X_A(v):= \tfrac{d}{dt}\big\vert_{t=0} \exp(tA)\cdot v= A \cdot v.
\] 
Notice that for an initial value $v_0 \in V$ the
curve $v(t) :=\exp(tA)\cdot v_0$ is the corresponding integral curve of $X_A$.
Recall also, that we denoted by $\ggo =\kg \oplus \pg$ the Cartan decomposition
of the Lie algebra $\ggo $ of $\G$. Then, for $A \in \kg$ the vector fields $X_A$ are Killing fields,
meaning that their flows consist of isometries.

For  fixed \mbox{$A \in \pg$} and $v \in V$ we let 
$$
 d (t):=d_{A,v}(t):=\Vert{ \exp(tA)\cdot v} \Vert^2
$$
denote the square of the distance function to the origin along $\exp(tA)\cdot v$.

\begin{lemma}[Convexity of the distance function]\label{lem_convex}
Let \mbox{$A \in \pg$} and $v \in V$ be given. 
Then $d'(0)= 2\cdot \langle A \cdot v ,v\rangle$
and  $ d''(t)= 4 \cdot \Vert A  \cdot \exp(tA)\cdot v \Vert^2$.
\end{lemma}

\begin{proof}
We have $d'(t)=2 \cdot \langle A \cdot \exp(tA)\cdot v, \, \exp(tA)\cdot v \rangle$.
From this the claim follows immediately using that $A^t = A$.
\end{proof}

\begin{corollary}\label{cor_limitlowernorm}
Let $A \in \pg$, $v \in V$ and suppose that $\lim_{t \to \infty} \exp(tA)\cdot v=
\bar v \neq v$ exists.
Then for all $t \in \RR$ one has $\Vert \exp(tA) \cdot v\Vert >\Vert \bar v\Vert$
for all $t \in \RR$.
\end{corollary}

Let us mention, that
 for a fixed $v$, the function $\G\to \RR$, $\exp(tA) \mapsto d(t)$ is usually called a \emph{Kempf-Ness function} in the literature.

Next, set $\T = \exp(\tg)$, with $\tg \subset \Sym(V,\ip)$ abelian, and let
\[
    \mca_\T := \{ v\in V : \Vert v \Vert \leq \Vert t \cdot v\Vert \hbox{ for all } t\in \T\}
\]
denote the set of minimal vectors for the $\T$-action. 
For any $v\in \mca_\T$ the orbit 
$\T\cdot v$ is closed by Lemma \ref{lem_oneparamT} and Corollary \ref{cor_limitlowernorm}. Conversely, for a closed $\T$-orbit $\oca$, the closest point to the origin in $\oca$ belongs to $\mca_\T$. 

Notice that by Lemma \ref{lem_convex} the condition of $v\in V$ being the closest point to the origin of $\T\cdot v$ is equivalent to
$ \la \lambda\cdot v, v\ra = 0$  for all $\lambda \in \tg$.
Since this condition is linear in $\lambda$ and polynomial in $v$,
 $\mca_\T$ is a closed subset of $V$. 
 
 Next, we show that the continuous, $\T$-invariant map $\Phi$ defined in
 the proof of Proposition \ref{prop_sepclosedorbits} is a proper map.

\begin{lemma}\label{lem_Phiproper}
There exists $C >0$ such that $\Vert v \Vert \leq C \cdot \Vert\Phi(v)\Vert$ for all $v\in \mca_\T$.
\end{lemma}

\begin{proof}
Recall that $\Phi(c\cdot v) = c \cdot \Phi(v)$ for all $c>0$ and $v\in V$. Assume
 that there exists a sequence $(v_k)_{k\in \NN} \subset \mca_\T$, $\Vert v_k\Vert \equiv 1$, with $\lim_{k \to \infty}\Phi(v_k) = 0$. For a subsequential limit $\bar v \in \mca_\T$, $\Vert \bar v\Vert=1$, we have $\Phi(\bar v) = 0$. But the orbit $\T \cdot \bar v$ is closed and non-trivial, 
 hence contained in some $U_I$.
 As a consequence, there exists one function $f\in \fca$ with $f(\T \cdot \bar v) > 0$. 
 But this contradicts $\Phi(\bar v) = 0$.
\end{proof}

\begin{corollary}\label{cor_nullconeclosedT}
The nullcone $\{ v\in V : 0 \in \overline{\T \cdot v} \}$ is a closed subset.
\end{corollary}

\begin{proof}
We have that $0\in \overline{\T\cdot v}$ if and only if $\Phi(v) = 0$.
\end{proof}

\begin{corollary}[Separation of closed $\T$-invariant subsets]\label{cor_sepclosedTinvsets}
Let $Z_1, Z_2 \subset V$ be two closed, disjoint, $\T$-invariant subsets. Then, there exists a continuous $\T$-invariant function $f: V \to [0,1]$ such that $f|_{Z_1} \equiv 0$ and 
$f\vert_{Z_2} \equiv 1$. 
\end{corollary}

\begin{proof}
By Urysohn's Lemma it 
is enough to show that $A_1:=\Phi(Z_1)$, \mbox{$A_2:=\Phi(Z_2)$} are closed, disjoint subsets,
since then we can set $f:=d \circ \Phi$, where $d:\RR^L \to [0,1]$ is a continuous function
with $d\vert_{A_1}\equiv 0$ and $d\vert_{A_2}\equiv 1$: see e.g. \cite[Ch.~I, Lemma 10.2]{Bred93}.

To see that $\Phi(Z_1)$ is closed consider a sequence $(\Phi(v_k))_{k\in \NN} \subset \Phi(Z_1)$ converging to some $\Phi_0 \in \RR^L$. Since $Z_1$ is closed and $\T$-invariant, we may assume that $v_k \in \mca_\T$ for all $k$ (recall that any $\T$-orbit has a closed $\T$-orbit in its closure by Corollary \ref{cor_oneclosedorbit}, and that $\Phi$ is continuous and $\T$-invariant). By Lemma \ref{lem_Phiproper} we then have that $(v_k)$ is bounded, thus it subconverges to some $\bar v \in Z_1$. Now $\Phi_0 = \Phi(\bar v)$, as we wanted to show. Clearly also $\Phi(Z_2)$ is closed. 

If $v_1 \in Z_1$, $v_2 \in Z_2$ are such that $\Phi(v_1) = \Phi(v_2)$, then as above we may assume that $v_1, v_2 \in \mca$, so that the corresponding $\T$-orbits are closed. But this contradicts Proposition \ref{prop_sepclosedorbits}. 
\end{proof}

\section{The general case of real reductive groups}\label{sec_generalactions}

We now focus on proving Theorem \ref{thm_realGIT}. The idea is to reduce it to the abelian case, already settled above. More precisely, let us fix $\tg \subset \pg$ a maximal abelian subalgebra, and let $\T := \exp(\tg)$ be the corresponding connected abelian Lie subgroup of $\G$. 
It will be proved in  Corollay \ref{cor_KTK} that one has
$\G = \K \T \K$, which in some sense says that  the non-compactness in $\G$ is abelian. 

We aim to proving that orbits containing minimal vectors are closed. Recall that we only consider in $V$ the standard vector space topology. Using the convexity of orbits of one-parameter subgroups (Lemma \ref{lem_convex}),  as a first step we prove the following


\begin{lemma}\label{lem_minlocclosed}
Let $\vmin \in \mca \subset V$ be a minimal vector with $\Vert \vmin\Vert=1$ and assume that $G\cdot \vmin$ is not closed. Then, there exists $\epsilon=\epsilon_{\vmin}>0$, such that  $\Vert v\Vert \geq 1+\epsilon$ for any $v \in \overline{G\cdot \vmin}\backslash G\cdot \vmin$.
\end{lemma}

\begin{proof}
As we will show below $\K \cdot \vmin$ admits an open, bounded
neighborhood $U$ in $\G\cdot \vmin$ such that the following holds:
the closure $\overline{U}$ of $U$
in $V$ satisfies $\overline U \subset \G\cdot \vmin$
and  there exists  $\epsilon>0$, such that for all
$v \in \G\cdot \vmin \backslash U$ we have $\Vert v \Vert \geq 1 + \epsilon$.
It follows then that any $v\in \overline{\G\cdot \vmin}\backslash \G\cdot \vmin$
satisfies $\Vert v\Vert \geq 1 +\epsilon$.

To show this claim,
let $\ggo_{\vmin}\subset \ggo$ denote the isotropy subalgebra of  $\vmin$ and $\pg^\perp_{\vmin}$
the orthogonal complement of $\ggo_{\vmin}\cap \pg$ 
in $\pg$ with respect to the given scalar product on $\ggo$. Then $\psi:\pg^\perp_{\vmin}\to \G \cdot \vmin\,;\,\,
A \mapsto \exp(A)\cdot \vmin$ is a local diffeomorphism
close to $0 \in \pg^\perp_{\vmin}$,
 such that its image intersects $\K\cdot \vmin$
transversally. Most importantly, by Lemma \ref{lem_convex} assuming
that $\Vert A\Vert=1$ we know that the  function 
$d(t)=d_{A,\vmin}(t)$ along $\exp(t\cdot A)\cdot \vmin$ satisfies
$d'(0)=0$ and $d''(0)=\Vert A \cdot \vmin\Vert^2$.
Since $A \in \pg^\perp_{\vmin}$ and $\Vert A\Vert =1$
there exists $\delta(\vmin)>0$ such that \mbox{$d''(0)\geq \delta(\vmin) >0$ }
for all such $A$.
Since $d''(t)=\Vert A \cdot \exp(t\cdot A) \cdot \vmin\Vert^2$ we deduce
furthermore,
that there exist $t_{\delta(\vmin )}>0$ and $\epsilon_{\vmin}>0$ such that 
for all $t \in \RR$ with $\vert t\vert\geq t_{\delta(\vmin)}$ 
and for all
 $A \in \pg^\perp_{\vmin}$ with $\Vert A\Vert =1$
we have $d_{A,\vmin}(t)\geq (1+\epsilon_{\vmin})^2$.

We consider now the  map $\Psi:\pg^\perp_{\vmin}\times U_{\vmin}^{\K\cdot \vmin}\to \G \cdot \vmin\,;\,\,(A,v) \mapsto \exp(A)\cdot v$ for an open neighbourhood
$U_{\vmin}^{\K\cdot \vmin}$ of $\vmin$ in $\K\cdot \vmin$. Again,
we may assume that $\Psi$ is a locall diffeomorphism from
$B_{t_\delta(\vmin)}(0) \times U_{\vmin}^{\K\cdot \vmin}$
to its image $U_{\vmin} \subset \G\cdot v$. Precisely as above, we deduce 
that there exist $t_{\delta(\vmin )}>0$ and $\epsilon_{\vmin}>0$ such that 
for all $t \in \RR\backslash (-t_{\delta(\vmin)},t_{\delta(\vmin)})$ and for all
 $A \in \pg^\perp_{\vmin}$ with $\Vert A\Vert =1$
 and for all $v \in U_{\vmin}^{\K\cdot \vmin}$ 
we have that $d_{A,v}(t)\geq (1+\epsilon_{\vmin})^2$.

Recall  that $\K \cdot \vmin \subset \mca$, since $\K$ acts isometrically. 
Since $\K\cdot \vmin$ is compact, there exist finitely many
 such open neighbourhoods 
$U_{k_1\cdot \vmin}, \ldots ,U_{k_N\cdot \vmin}$, such that 
$U=\cup_{i=1}^N U_{k_i\cdot \vmin}$ is an open neighborhood of $\K\cdot \vmin$, $k_1=e,\ldots ,k_N \in \K$.
We may of course assume that
each of the open subsets $U_{k_i\cdot \vmin}^{\K\cdot \vmin}$ contains
a compact subset $A_{k_i\cdot \vmin}^{\K\cdot \vmin}$, $i=1, \ldots ,N$,
such that the interior of these sets still cover $\K \cdot \vmin$, that is
 $\K\cdot \vmin\subset U^\K
 :=\cup_{i=1}^N (A_{k_i\cdot \vmin}^{\K\cdot \vmin})^o$.
This then shows the above claim.
\end{proof}

In the next lemma we show that a non-closed orbit has a closed, $G$-invariant subset in its closure, intersecting the orbit trivially. 

\begin{lemma}\label{lem_niceclosurepoint}
Let $v\in V$ and suppose that the orbit $\G\cdot v$ is not closed. Then
there exists $\bar v \in  \overline{\G \cdot v} \backslash \G\cdot v$, such that
$Y=\overline{G\cdot \bar v}$ satisfies $Y \cap \G \cdot v =\emptyset$.
\end{lemma}

\begin{proof}
Notice first that the set  $Z:=\overline{\G \cdot v}$ is closed
and $\G$-invariant, hence it contains
a minimal vector $\vmin \in \mca$.
 Since by assumption $\G\cdot v$ is not closed, there are now two cases to be considered:  $\vmin \in \G\cdot v$ and $\vmin \in Z \backslash \G\cdot v$.
 
 If $\vmin \in \G\cdot v$ by Lemma \ref{lem_minlocclosed} there
 exists $\epsilon>0$ such that for any $\bar v \in Z\backslash \G\cdot v$ 
 we have $\Vert \bar v\Vert \geq \Vert \vmin\Vert + \epsilon$.
 In particular, the same estimate holds true on
 $Y=\overline{\G\cdot \bar v}$.  Since $\G\cdot v \cap Y$ is $\G$-invariant,
 this intersection must be empty, since otherwise it would contain $\G \cdot v$,
 hence $\vmin$ contradicting the above estimate.
 
 In case  $\vmin \in Z \backslash \G\cdot v$, 
 we set $Y:= \overline{\G\cdot \vmin}$.
  Again, by Lemma \ref{lem_minlocclosed} any element in 
  $\overline{\G\cdot \vmin}\backslash \G\cdot  \vmin$ must 
  satisfy $\Vert v\Vert \geq \Vert \vmin\Vert + \epsilon$ for some 
  $\epsilon>0$. Thus $\G\cdot v \cap Y=\emptyset$,
  since $\vmin \in \overline{\G\cdot v}$ implies that  $\G\cdot v$ must contain
  vectors of norm $ \Vert \vmin\Vert + \epsilon/2$.
\end{proof}

We can now provide a proof of the Hilbert-Mumford criterion in this setting. The following argument is due to  Richardson (see also \cite[Thm.~ 5.2]{Birkes}).

\begin{lemma}[Hilbert-Mumford criterion for real reductive groups]\label{lem_oneparamG}
Let $v\in V$. If the orbit $\G\cdot v$ is not closed, then for some $\alpha\in \pg$ the limit
 $ \lim_{t\to\infty} \exp(t \alpha) \cdot v$ exists.
\end{lemma}

\begin{proof}
Let $\T = \exp(\tg)\subset \G$ be a maximal abelian subalgebra, $\tg \subset \pg$ , and choose  $\bar v \in \overline{\G \cdot v} \backslash \G\cdot v$ such that
$Y:=\overline{\G \cdot \bar v}$ satisfies
$Y \cap \G \cdot v=\emptyset$: see \mbox{Lemma \ref{lem_niceclosurepoint}}.
We will show below that then there exists 
$g \in  \G$, $k\in \K$ and $\alpha \in \tg$ such that $g \cdot \bar v = \lim_{t\to\infty} \exp(t \alpha) \cdot (k\cdot v)$. Notice that the lemma follows, 
since for $\alpha' = k^{-1} \, \alpha \, k$, $\bar v ' = k^{-1} \, \cdot g \cdot \bar v$ 
we deduce $\bar v' = \lim_{t\to\infty} \exp(t \alpha') \cdot v$. 

To prove the above claim, suppose on the contrary that $Y \cap \overline{\T \cdot k \cdot v} = \emptyset$ for all \mbox{$k\in \K$}. Since $Y$ is closed and $\T$-invariant,
by Corollary \ref{cor_sepclosedTinvsets}, for each $k\in \K$ there exists a continuous $\T$-invariant function $f_k : V \to \RR$ 
with $f_k\big(\overline{\T\cdot (k\cdot v)}\big) = 1$ and
$ f_k\big(Y) \equiv 0$. 
By continuity, each $k$ has an open neighborhood $U_k$ in $\K$ such that $f_k(\T \cdot U_k \cdot v) > 1/2$. Since $\K$ is compact, we may extract a finite number of such functions $f_{k_1}, \ldots, f_{k_R}$ such that for $f = f_{k_1} + \cdots + f_{k_R}$ we have that 
$f\big((\T\K) \cdot v\big) > 1/2$ and $ f\big( Y) \equiv 0$.
 Since $K\cdot \bar v \subset Y$, we deduce
 $\overline{\T\K \cdot v} \cap \K \cdot \bar v = \emptyset$, thus $\bar v \notin \K \left(\overline{\T\K \cdot v} \right)$.
Using $\G =\K \T \K$ and $\overline{\G\cdot v} \subset
 \K \left(\overline{\T\K \cdot v} \right)$, we obtain $\bar v \not\in
 \overline{\G\cdot v}$, a contradiction. To see why one has $\overline{\G\cdot v} \subset
 \K \left(\overline{\T\K \cdot v} \right)$, observe that if $w = \lim_{i\to\infty} g_i \cdot v$ with $g_i = k_i t_i k'_i \in \G = \K \T \K$, by compactness of $\K$ one may assume that $k_i \to k_\infty$ and hence $w = k_\infty  \cdot \lim_{i\to\infty} t_i k'_i \cdot v \in \K \,  (\overline{T K \cdot v})$.
\end{proof}

\begin{corollary}\label{cor_closediffminimal}
Let $v\in V$. Then, the orbit $\G\cdot v$ is closed if and only if there exists $v_m \in \G\cdot v$ with $\Vert v_m\Vert \leq \Vert g \cdot v_m\Vert$ for all $g\in \G$.
\end{corollary}

\begin{proof}
If the orbit is closed then the existence of $v_m$ is clear. Conversely, assume that there exists a minimal vector $v_m$ but $\G \cdot v_m$ is not closed. Notice that by continuity we also have that $ \Vert v_m\Vert \leq \Vert \bar v\Vert$ for all $\bar v \in \overline{ \G\cdot v_m}$. Lemma \ref{lem_oneparamG} together with Corollary \ref{cor_limitlowernorm} give a contradiction.
\end{proof}

\begin{lemma}\label{lem_uniqclosedorbit}
Any orbit $\G\cdot v$ contains exactly one closed orbit in its closure.
\end{lemma}

\begin{proof}
Suppose that $\bar v \in \overline{\G\cdot v}\backslash \G \cdot v$ 
has minimal norm. Then, by Corollary \ref{cor_closediffminimal} the orbit $\G\cdot \bar v \subset \overline{\G\cdot v}$ is closed.  Suppose furthermore, 
that there is second closed orbit $\G\cdot \bar w \subset \overline{\G\cdot v}\backslash \G\cdot \bar v$.
 Then, by Corollary \ref{cor_sepclosedTinvsets} there exists a continuous, $\T$-invariant function  $f:V \to [0,1]$ with $f(\G\cdot \bar v)=1$  and $f(\G \cdot \bar w)=0$.
Let now $(w_i)_{i\in \NN} \subset \G\cdot v$ 
be a sequence with $\lim_{i \to \infty} w_i = \bar w$. By the claim in the proof of Lemma \ref{lem_oneparamG}, for each $i$ there exist $g_i \in \G$, $k_i \in \K$ and $\alpha_i \in \tg$ such that $\lim_{t\to\infty} \exp(t \alpha_i) \cdot (k_i \cdot w_i) = g_i \cdot \bar v$. Since $f$ is $\T$-invariant and continuous,
we deduce $f(k_i \cdot w_i)=1$ for all $i \in \NN$. On the other hand side,
the sequence $(k_i \cdot w_i)$ subconverges to a vector in $\G\cdot \bar w$,
hence $f(k_i \cdot w_i) \to 0$ along that subsequence. Contradiction.
\end{proof}

\begin{proof}[Proof of Theorem \ref{thm_realGIT}] For (i):  if $v$ is a minimal vector then by Lemma \ref{lem_convex} and the decomposition $\G = \K \T \K$
 the set of minimal vectors in the closed orbit $\G\cdot v$ is precisely $\K \cdot v$. 
For (ii): let $\G\cdot v$ be a non-closed orbit, and pick $\bar v \in \overline{\G\cdot v}$ of minimal norm. Then, by Corollary \ref{cor_closediffminimal}  the orbit $\G \cdot \bar v$ is closed. From the proof of Lemma \ref{lem_oneparamG} we know that there exists a one-parameter subgroup such that $\lim_{t\to\infty}\exp(t\alpha') \cdot v 
\in \G \cdot \bar v$.
The third item is precisely Lemma \ref{lem_uniqclosedorbit}. 
For (iv):  Let $(v_i) \subset V$ be a sequence with $0 \in \overline{\G\cdot v_i}$ for all $i$ such that $v_i \to v_\infty$. By (ii), it follows that there exists maximal abelian subalgebras $\tg_i \subset \pg$ such that $0\in \overline{\T_i\cdot v_i}$, where $\T_i := \exp(\tg_i)$. Since by 
Proposition \ref{prop_conjugate}  all such $\T_i$ are conjugate by elements in $\K$, 
we may assume (after possibly changing $v_\infty$ by $k\cdot v_\infty$, $k\in \K$) that $\T_i \equiv \T$ is constant. The result now follows from Corollary \ref{cor_nullconeclosedT}.
\end{proof}

\section{Stratification}\label{sec_strat}

In this section we provide a proof for Theorem \ref{thm_stratif}.
This will be done by using the energy map
$\normmm(v)=\Vert \mmm(v)\Vert^2$ associated to the moment map $\mmm:V\backslash \{0\} \to\pg$ (see \eqref{eqn_defmm}) as a Morse function. The map $\normmm$  has the following remarkable property: 
its critical points are mapped under the moment map onto finitely many $\K$-orbits 
$\K\cdot \Beta_1, \ldots ,K\cdot \Beta_N$ in $\pg$ (see Lemma \ref{lem_finiteness}). We set $\bca :=\{\Beta_1, \ldots ,\Beta_N\}$,
and for $\Beta \in \bca$ we let $\cca_\Beta$ denote the set of critical points of
$\normmm$ with $\mmm(\cca_\Beta) \subset \K\cdot \Beta$. It will turn out that the stratum
$\sca_\Beta \subset V\backslash \{0\}$ is the unstable manifold corresponding to $\cca_\Beta$.

In order to briefly describe how the strata are constructed, let us fix $\Beta \in \bca$ and let $\Gb$ denote the centralizer of $\Beta$ in $\G$. It turns out that critical points $v_C$ of $\normmm$ with $\mmm(v_C) = \Beta$ correspond to minimal vectors for the action of a real reductive subgroup $\Hb \subset \Gb$ with Lie algebra $\hg_\Beta=\Beta^\perp$ on a certain subspace $\Vzero \subset V$. This makes it possible to apply Theorem \ref{thm_realGIT} for the restricted action. Inspired by the negative directions of the Hessian of $\normmm$ at $v_C$ (Lemma \ref{lem_Hess}) and the fact that $\normmm$ is $\K$-invariant, one defines the stratum $\sca_\Beta$ as in Definition \ref{def_Sbeta}. After proving that this is a smooth submanifold (Proposition \ref{prop_scasmooth}), it will follow that $\sca_\Beta$ is invariant under the negative gradient flow of $\normmm$.   

Along the proof it will turn out to be extremely convenient to break the $\K$-symmetry, and work with a fixed $\Beta$ as opposed to the entire $\K$-orbit $\K\cdot \Beta$. Thus, the crucial results will be proved on the \emph{slice} $\Vnn$ (see \eqref{eqn_defVnn}), and then extended to all of $\sca_\Beta$ by $\K$-invariance. This forces us to work with a certain \emph{parabolic subgroup} $\Qb$ of $\G$ associated to $\Beta$ (Definition \ref{def_groups}) which preserves the subspace $\Vnn$. Well-known properties of $\Qb$ and of other subgroups of $\G$ adapted to $\Beta$ will be needed along this section. They will be proved in Appendix \ref{app_groups}.

Recall that $\G\subset \Gl(V)$ is a closed subgroup satisfying \eqref{eqn_Cartandec} (see also Appendix \ref{app_reductive}). In particular, 
$\ggo \subset \glg(V)$ is a Lie subalgebra, with the property that $A^t \in \ggo$ for all $A\in \ggo$. Recall also that $\ggo$ has a Cartan decomposition $\ggo = \kg \oplus \pg$, 
where $\kg\subset \sog(V)$ and $\pg \subset \Sym(V)$.

\begin{notation}
For $\Beta \in \pg$ we set $\Beta^+ := \Beta - \Vert \Beta \Vert^2 \cdot \Id_V \in \Sym(V)$.
\end{notation}

This notation appears naturally in the formula for the gradient of $\normmm$:

\begin{lemma}\label{lem_gradient} 
The gradient of the energy map $\normmm : V\zero \to \RR$ 
is given by
\[
    (\nabla \normmm)_v  
    = \tfrac4{\Vert{v}\Vert^2} \cdot {\mmm(v)^+ } \cdot v \,.
\]
\end{lemma}

\begin{proof}
Since $\normmm$ is scaling-invariant,  we have
$ (\nabla \normmm)_v \perp v$ for $v \in V\zero$.
So let $w \perp v$ with $\Vert v \Vert=\Vert w \Vert$ and
set $v(t)=\cos(t)v+\sin(t)w$. Then by (\ref{eqn_defmm}) for $\Alpha \in \pg$ we have 
\begin{eqnarray}
    \big\langle (d\mmm)_{v(t)}\, v'(t), \Alpha \big\rangle 
    = \tfrac2{\Vert{v}\Vert^2} \cdot\big\la \Alpha \cdot v(t), v'(t)\big\ra\,.
    \label{def_diffmoment}
\end{eqnarray}
Thus $\langle (\nabla \normmm)_v ,w \rangle = \tfrac4{\Vert{v}\Vert^2} \cdot \la {\mmm(v)} \cdot  v, w\ra$ and the lemma follows
since again by (\ref{eqn_defmm}) we have 
$\tfrac1{\Vert{v}\Vert^2}\la {\mmm(v)}  \cdot v, v\ra \cdot v=
\Vert {\mmm(v)} \Vert^2\cdot v$.
\end{proof}

Since $\Beta\in \pg$ is a symmetric endomorphism,
we may decompose $V$ as a sum of eigenspaces $\Vr$ of $\Beta^+$ corresponding to its eigenvalues $r\in \RR$. Of major important will be  $\Vzero$, the kernel of $\Beta^+$, and the sum of the non-negative eigenspaces
\begin{equation}\label{eqn_defVnn}
   \Vnn := \bigoplus_{r \geq 0} \Vr\,.
\end{equation}
The reason is that the above mentioned Hessian is non-negative on $\Vnn$
in every critical point $v_C\in\cca_\Beta$ of $\normmm$: see Lemma \ref{lem_Hess} below.
Let us explicitly mention though, that at this point $\Beta$ is arbitrary.

Analogous to the subspaces $\Vzero$, $V_{\Beta^+}^{>0} := \bigoplus_{r>0} \Vr$ and $\Vnn$, we have certain special subgroups of $\G$. To define them, consider the symmetric endomorphism
\[
 \ad(\Beta) : \ggo \to \ggo\,\,;\,\,\, A \mapsto [\Beta, A]\,.
\]
Using the eigenspace decomposition $\ggo = \bigoplus_{r\in \RR} \ggo_r$ of $\ad(\Beta)$, 
we denote  $\ker (\ad(\Beta) )$ by $\ggo_\Beta := \ggo_0 $, set
$\ug_\Beta := \bigoplus_{r> 0} \ggo_r$ and $\qg_\Beta = \ggo_\Beta \oplus \ug_\Beta$.

\begin{definition}\label{def_groups}
We denote by $\Gb := \{ g \in \G : g \Beta g^{-1} = \Beta \}$ the centralizer of $\Beta$ in $\G$, by $\Ub := \exp(\ug_\Beta)$, and we set $\Qb := \Gb \Ub$. 
\end{definition}
It turns out that $\Gb, \Ub$ and $\Qb$ are closed subgroups of $\G$, with Lie algebras $\ggo_\Beta$, $\ug_\Beta$ and $\qg_\Beta$, respectively. We refer the reader to Appendix \ref{app_groups} for more details and properties of these groups.

\begin{lemma}\label{lem_Vnninv}
The subspace $\Vnn$ is $\Qb$-invariant. 
\end{lemma}

\begin{proof}
For $A\in \ggo$, if the action field $X_A(v) = A \cdot v$ is tangent to a subspace $W$ of $V$,
then the integral curves of $X_A$ starting tangent to $W$ cannot leave $W$. 
Thus it suffices to show that for all $A\in \qg_\Beta$ and $v\in \Vnn$, we have that $A \cdot v \in \Vnn$.

By linearity we may assume that $v\in \Vr$, $r\geq 0$, and that $A$ is an eigenvector of $\ad(\Beta) : \ggo \to \ggo$ with eigenvalue $\lambda_A \geq 0$. Then, 
\begin{eqnarray}
    \Beta^+ \cdot (A\cdot v) = [\Beta^+, A] \cdot v + A \cdot \Beta^+ \cdot v = \lambda_A  (A \cdot v) + r (A \cdot v),\label{eqn_betapA}
\end{eqnarray}
thus $A\cdot v \in V_{\Beta^+}^{r+\lambda_A} \subset \Vnn$ and the lemma follows.
\end{proof}

The linear orthogonal projection $p_\Beta: \Vnn \to \Vzero$ 
will be  important later on. 

\begin{lemma}\label{lem_pbeta}
The orthogonal projection $p_\Beta : \Vnn \to \Vzero$ satisfies the formula
\begin{equation}\label{eqn_pbeta}
    p_\Beta(v) = \lim_{t\to \infty} \exp(-t \Beta^+) \cdot v,
\end{equation}
it is  $\Gb$-equivariant, and for each $v\in \Vzero$ 
the fibre $p_\Beta^{-1}(v)$ is $\Ub$-invariant.
\end{lemma}

\begin{proof}
Let $v = \sum_{r\geq 0} v_r$ with $v^r \in \Vr$.
Since the action of $\exp(-t \Beta^+)$ on $\Vr$ 
is simply given by scalar multiplication by $e^{- t r}$, we immediately obtain (\ref{eqn_pbeta}). To show the $\Gb$-equivariance, 
let $v \in \Vnn$ and $h\in \Gb$. Since
$[\Gb,\exp(\Beta^+)]=0$ we deduce
\[
    p_\Beta(h \cdot v) = \lim_{t\to\infty} \exp(-t\Beta^+) \cdot (h \cdot v) = h\cdot  \lim_{t\to\infty} \exp(-t \Beta^+) \cdot v = h \cdot p_\Beta(v)\,.
\]
To prove that the above fibre  $p_\Beta^{-1}(v)$ is $\Ub$-invariant,
recall that for $u\in \Ub$ we have
$ \lim_{ t \to \infty}\exp(-t \Beta^+) \cdot u \cdot \exp(t \Beta^+) =e$  by Lemma \ref{lem_groupslimit} .
It follows that
\[
    p_\Beta(u \cdot v) = \lim_{t\to\infty} \exp(-t \Beta^+) \cdot u \cdot \exp(t \Beta^+) \cdot \exp(-t \Beta^+) \cdot v = p_\Beta(v)\,,
\]
which shows the claim.
\end{proof}

\begin{remark}\label{rmk_pbeta}
From the proof of the previous lemma it also follows that for an arbitrary $v\in V$, the limit in \eqref{eqn_pbeta} exists if an only if $v\in \Vnn$.
\end{remark}

Before introducing the strata $\sca_\Beta$ algebraically 
we need to consider one further group
related to the $\Qb$-action on $\Vnn$. Recall that the group $\Gb$ is reductive, with Cartan decomposition given by $\Gb = \Kb \exp(\pg_\Beta)$ induced from that of $\G$: $\Kb = \Gb \cap \K$, $\pg_\Beta = \ggo_\Beta \cap \pg$. Consider the following Lie subalgebra of $\ggo_\Beta$
\[
    \hg_\Beta := \{ \Alpha \in \ggo_\Beta : \langle \Alpha  ,   \Beta \rangle = 0\}.
\]

\begin{definition}\label{def_groups2}
The subgroup  $\Hb \subset \Gb$ is defined by
\[
	\Hb = \Kb \, \exp(\pg_\Beta \cap \hg_\Beta).
\]
\end{definition}

The group $\Hb$ is real reductive, see \eqref{eqn_Cartandec}. Its Lie algebra is $\hg_\Beta$ and we have
 $\Gb = \exp(\RR \Beta) \times \Hb$ by the explicit description of $\Gb$
 given after Definition \ref{def_groupsapp}.
Moreover, it follows from (\ref{eqn_betapA}),
 that $\Hb$ acts on $\Vzero$, and that this action satisfies \eqref{eqn_Cartandec} with respect to the induced scalar product on $\Vzero$. Thus Theorem \ref{thm_realGIT} applies in this case.

\begin{definition}\label{def_Sbeta}
We call
\[
 	\Vzeross := \big\{ v \in \Vzero : 0\notin \overline{\Hb\cdot v} \big\}\,
\]
the subset of $\Hb$-\emph{semistable} vectors in $\Vzero$.
We also define accordingly 
\[
 	\Vnnss := p_\Beta^{-1} \big(\Vzeross\big).
\] 
Then, the \emph{stratum} $\sca_\Beta$ associated with the orbit $\K\cdot\Beta$ 
is the set defined by 
\[
  \sca_\Beta := \G \cdot \Vnnss\,.
\]
\end{definition}

It will be made clear afterwards that for most $\Beta \in \pg$ the
stratum $\sca_\Beta$ is actually empty. However, if
the subset $\Vzeross$  of semi-stable vectors is non-empty, then it is an
open subset of $\Vzero$ by Theorem \ref{thm_realGIT}, (iv), applied to the action of $\Hb$ on $\Vzero$. The same holds of course for $\Vnnss$ in $\Vnn$. 


\begin{remark}\label{rmk_Sscaleinv}
Notice that the strata $\sca_\Beta$ are scale-invariant. Indeed, a vector $v\in \Vzero$ is $\Hb$-semi-stable if and only if $c \cdot v$ is so, for any $c\neq 0$. Thus, $\Vzeross$ is scale invariant, and the same holds for $\Vnnss$ since $p_\Beta$ is a linear map.
\end{remark}

A second observation is that for  a critical point $v_C$ of $\normmm$
we have $v_C \in \Vzeross \subset \sca_\Beta$, where $\Beta := \mmm(v_C)$. 
This also follows from Theorem \ref{thm_realGIT} applied to the action of $\Hb$ on $\Vzero$, since the following lemmas will imply that 
 the moment map 
\[
  \mmm_{\Hb} : \Vzero\backslash\{ 0\} \to \pg\cap \hg_\Beta
\] 
for this action, which a priori is given by the orthogonal projection of $\mmm(v)$ to $\hg_\Beta$, satisfies the formula 
\[
    \mmm_{\Hb}(v) = \mmm(v) - \Beta.
\]
Let us mention that up to the $\K$-action on $\pg$ there are only finitely many $\Beta$'s of the form $\mmm(v_C)$, for $v_C$ a critical point of the energy map $\normmm$ (Lemma \ref{lem_finiteness}).

The two main statements to be proved are the fact that the strata are smooth submanifolds, and that not only the critical point $v_C$, but also the entire flow lines of the negative gradient flow of $\normmm$ converging to $v_C$, are contained in the corresponding stratum $\sca_\Beta$. From this, the other assertions in Theorem \ref{thm_stratif} will easily follow.


We first compute the moment map on $\Vzero$.


\begin{lemma}\label{lem_mmZbeta}
For $v \in \Vzero\zero$ we have that $\mmm(v) = \mmm_{\Hb}(v) + \Beta$. Moreover, 
if  $\mmm(v) = \Beta$, then $v$ is a critical point of $\normmm$.
\end{lemma}

\begin{proof}
Let $v  \in \Vzero\zero$ with $\Vert v \Vert=1$. 
Since $\Beta \in \pg$ and $\Beta^+ \cdot v = 0$, \eqref{eqn_assumipggo} implies that for any $\Alpha \in \ggo$ we have that
\begin{align*}
    \la [\Beta,\mmm(v)], \Alpha\ra =& \la \mmm(v), [\Beta, \Alpha] \ra  =
     \la [\Beta, \Alpha] \cdot v, v\ra
    	=\left\la [\Beta^+, \Alpha] \cdot v, v\right\ra = 0\,.
\end{align*}
This shows that $\mmm(v) \in \ggo_\Beta$. Recall that $\ggo_\Beta = \RR \Beta \oplus \hg_\Beta$, and observe that
\[
    \la \mmm(v), \Beta \ra = \la \Beta \cdot v, v\ra = \la \Beta^+ \cdot v, v\ra + \Vert \Beta\Vert^2
    =\la \Beta,\Beta\ra\,.
\]
Hence the $\Beta$-component of $\mmm(v)$ is precisely $\Beta$. On the other hand, the orthogonal projection of $\mmm(v)$ to $\hg_\Beta$ is $\mmm_{\Hb}(v)$, since $\Hb \subset \G$.

The  last assertion follows from
Lemma \ref{lem_gradient}, since $\mmm(v)^+ \cdot v = \Beta^+ \cdot v = 0$.
\end{proof}

In the next step, we deduce that $\normmm|_{\Vnn}$ attains its minimum 
precisely at the critical points of $\normmm$ with critical value $\Vert\Beta\Vert^2$.

\begin{lemma}\label{lem_estimatemm}
For $v \in \Vnn \backslash \{0 \}$ we have that $\Vert{\mmm(v)}\Vert \geq \Vert{\Beta}\Vert$, with equality if and only if $v\in \Vzero$ is a critical point of $\normmm$ with $\mmm(v) = \Beta$.
\end{lemma}

\begin{proof}
We write $v = \sum_{r\geq 0} v_r$ with $v_r \in \Vr$ and $\Vert{v}\Vert^2=1$. Notice that $\Beta \cdot v_r = (r+\Vert \Beta \Vert^2) v_r$. Then, we have that
\[
	\la \mmm(v), \Beta \ra  = \sum_{r\geq 0} \la \Beta \cdot  v_r , v_r \ra \geq \Vert \Beta \Vert^2.
\]
It is clear that equality holds if and only if  $v \in \Vzero$. By Cauchy-Schwarz 
we deduce $\Vert{\mmm(v)}\Vert^2 \geq \la \mmm(v), \Beta\ra \geq \Vert{\Beta}\Vert^2$, with equality if and only if $\mmm(v) = \Beta$ and $v \in \Vzero$. By Lemma \ref{lem_mmZbeta} $v$ is then a critical point of $\normmm$. 
\end{proof}



The following lemma shows that critical points of $\normmm$ which are mapped to $\Beta$ under the moment map
 $\mmm$, correspond to minimal vectors for the $\Hb$-action on $\Vzero$. 



\begin{lemma}\label{lem_equivWbeta}
Let $v\in \Vzero$. Then, $v \in  \Vzeross$ if and only if there exists 
a critical point $v_C\in \overline{\Hb \cdot v} \backslash \{ 0\}$ 
of $\normmm$ with $\mmm(v_C) = \Beta$.
\end{lemma}

\begin{proof}
Clearly, $0\notin \overline{\Hb \cdot v}$ is equivalent to the existence of a vector $w \in \overline{\Hb \cdot v}$ of minimal positive norm. This implies that for all $ \Alpha \in \hg_\Beta$ we have
\[
    0=  \ddt \big|_0 \Vert \exp(t \Alpha) \cdot w \Vert^2
     =  2\cdot \la \Alpha \cdot  w \, , \, w\rangle
     =  2 \cdot\Vert w \Vert^2\cdot \la \mmm(w), \Alpha \ra\,. 
\]
Thus, $\mmm(w) \perp \hg_\Beta$. By Lemma \ref{lem_mmZbeta} we deduce that $\mmm(w) = \Beta$, and that $w$ is a critical point. 

Conversely, if $0 \in \overline {\Hb \cdot v}$ then there would be two closed $\Hb$-orbits in $\overline {\Hb \cdot v}$: the one corresponding to $v_C$, and $\{0\}$. This contradicts Theorem \ref{thm_realGIT}.
\end{proof}


In the next lemma we show, up to the action of $\K$, the stratum equals $\Vnnss$. 

\begin{lemma}\label{lem_WbetaQstable}
The set
$\Vnnss$ is $\Qb$-invariant and
$\sca_\Beta = \K \cdot \Vnnss$.
\end{lemma}

\begin{proof}
The group $\Hb$ leaves $\Vzeross$ invariant by definition. Since $\Vzeross \subset \Vzero$, and the latter is an eigenspace of $\exp(\Beta)$, it follows that $\Gb = \exp(\RR \Beta) \times \Hb$ preserves $\Vzeross$, because an orbit is closed if and only if the scaled orbit is so.  By its definition and Lemma \ref{lem_pbeta}
 we have that $\Vnnss$ is $\Gb$-invariant and  $\Ub$-invariant. 
 The last assertion follows from the facts that  $\sca_\Beta= \G \cdot \Vnnss$ and $\G=\K \Qb$: see Lemma \ref{lem_groupsGGbeta}.
\end{proof}

\begin{corollary}\label{cor_estimatemmm}
For $v \in \overline{\sca_\Beta} \backslash \{ 0\}$ we have that $\Vert {\mmm(v)} \Vert \geq \Vert{\Beta} \Vert$. Equality holds if and only if $\mmm(v) \in \K\cdot \Beta$, and in this case $v\in \sca_\Beta$ is a critical point for $\normmm$. 
\end{corollary}
\begin{proof}
From Lemma \ref{lem_WbetaQstable} and compactness of $\K$ we have that $\overline{\sca_\Beta} \subset \K \cdot \overline{\Vnnss} \subset \K \cdot \Vnn$. The claim now follows from Lemma \ref{lem_estimatemm} and the $\K$-equivariance (resp. invariance) of $\mmm$ (resp. $\normmm$).
\end{proof}

This shows that on the stratum $\sca_\Beta$ the energy 
map $\normmm$ is bounded below by the critical value $\normmm(v_C)$, for $v_C \in \mmm^{-1}(\Beta)$. Moreover, the minimum is attained precisely at those
critical points $v_C$ of $\normmm$  with $\mmm(v_C)\in \K\cdot\Beta$.




\begin{corollary}\label{cor_closureGmu}
If $v \in \sca_\Beta$ then there exists a critical point
$v_C \in  \overline{( \RR_{>0} \cdot \G) \cdot v} \subset \overline{\sca_\Beta}$  of $\normmm$ with $\mmm(v_C) = \Beta$.  
\end{corollary}

\begin{proof}
By Lemma \ref{lem_WbetaQstable} we may replace $v$ by $k\cdot v \in \Vnnss$, 
for some $k\in \K$. Applying Lemma \ref{lem_equivWbeta} to $p_\Beta(k\cdot v) \in \Vzeross$ and using Lemma \ref{lem_pbeta}
and (\ref{eqn_pbeta}) one obtains 
\[
    v_C \in   \overline{ \Hb \cdot p_\Beta(k\cdot v) }  =  \overline{ p_\Beta (\Hb \cdot k \cdot v) } \subset  \overline{(\RR_{>0} \cdot \G) \cdot v}
\]
as in the statement.
\end{proof}

 

We are now in a position to prove Theorem \ref{thm_stratif}. The key analytical property is the
 negativity of the Hessian of the energy map $\normmm$ restricted to
the normal space of a stratum at a critical point: see Lemma \ref{lem_Hess}.

\begin{proof}[Proof of Theorem \ref{thm_stratif}]
For each critical point $v_C$ of $\normmm$ we consider the set $\sca_\Beta$ defined in Definition \ref{def_Sbeta}, where $\Beta = \mmm(v_C)$.
Notice first that  for all $k\in \K$,
$\sca_{k \cdot \Beta \cdot k^{-1}} = \sca_\Beta$ by $\K$-equivariance. 
By picking one representative for each $\K$-orbit, say a diagonal $\Beta$ with eigenvalues in non-decreasing order, the strata may be parameterized by a finite set $\bca$ by Lemma \ref{lem_finiteness}.

We now prove one direction in (iii). Let $v(t)$ denote a solution to
the negative gradient flow of $\normmm$ with $v(0)=v \in V\zero$.
Since $\normmm$ is scale-invariant we have $\Vert {v(t)} \Vert \equiv \Vert {v}\Vert$.
Using that $\normmm$ is real analytic, by {\L}ojasiewicz' theorem \cite{Loj63} there exists a unique
limit point $\lim_{t \to \infty}v(t)=v_C$, which is of course
a critical point of $\normmm$.{}
 Let $\Beta := \mmm(v_C)$ and notice that $v_C \in \Vzeross \subset \sca_\Beta$ by Lemma \ref{lem_equivWbeta}. 
Since $\sca_\Beta$ is a smooth embedded submanifold of $V$ by Proposition \ref{prop_scasmooth}, there exists an open neighbourhood $\Omega\subset V$ of $v_C$ which is diffeomorphic to the normal bundle of $\sca_\Beta$ restricted to $\sca_\Beta \cap \Omega$, such that $\sca_\Beta \cap \Omega$ is the zero section. For some $t_0 \in \RR$ we  have $v(t) \in \Omega$, $\forall t\geq t_0$. Also, since $v_C \in \Vnn \cap \sca_\Beta$, we obtain
$T_{v_C} \sca_\Beta = \kg \cdot  v_C+ \Vnn  $.
By the Hessian computations from Lemma \ref{lem_Hess} and a standard second-order argument, see Remark \ref{rem_secord}, we conclude that we must have $v(t) \in \sca_\Beta \cap \tilde \Omega$ for all $t\geq t_0$
and some open subset $\tilde \Omega \subset \Omega$. 
 Since the flow lines are, up to scaling, tangent to $\G$-orbits,
and since the stratum $\sca_\Beta$ is $\G$-invariant and scale-invariant (Remark \ref{rmk_Sscaleinv}),
we conclude that $v \in \sca_\Beta$ is  as well. As a consequence
$V\backslash \{0\}=\bigcup_{\Beta \in \bca} \sca_\Beta$.

To prove (i) it remains to show that $\sca_\Beta \cap \sca_{\Beta'} \neq \emptyset$ implies $\K \cdot \Beta = \K \cdot \Beta'$. Suppose that $v \in \sca_\Beta\cap \sca_{\Beta'}$.  By Corollary \ref{cor_closureGmu} one obtains 
critical points $v_C \in \overline{\sca_\Beta}$ 
and $v_C' \in \overline{\sca_{\Beta'}}$ 
with $\mmm(v_C) = \Beta$ and $\mmm(v_C') = \Beta'$. Since $v_C,v_C' \in 
\overline{( \RR_{>0} \cdot \G) \cdot v}$,
we deduce $v_C,v_C' \in \overline{ \sca_\Beta \cap \sca_{\Beta'}}$. 
Applying Corollary \ref{cor_estimatemmm} twice we get $\Vert{\Beta}\Vert = \Vert{\Beta'}\Vert$, thus by the rigidity in the equality case in that result we conclude that $\K\cdot \Beta =\K \cdot \Beta'$.

We can now prove the other direction in (iii). Let $v \in \sca_\Beta$ and assume that the limit $v_C$ of the negative gradient flow of $\normmm$ starting at $v$ satisfies $\mmm(v_C) = \Beta'$. By the above we have $v \in \sca_{\Beta'}$, thus $\sca_\Beta \cap \sca_{\Beta'} \neq \emptyset$ and hence $\Beta' \in \K \cdot \Beta$ and $\sca_\Beta = \sca_{\Beta'}$.

Finally, to show (ii) let $v \in \overline{\sca_\Beta} \backslash \sca_\Beta$, with say $v \in \sca_{\Beta'}$. By (iii) we may assume that $\mmm(v) = \Beta'$, thus Corolary \ref{cor_estimatemmm} applied to $\overline{\sca_\Beta}$ yields $\Vert{\Beta'}\Vert = \Vert{\mmm(v)}\Vert \geq \Vert{\Beta}\Vert$. Since equality would imply that $\Beta' \in \K \cdot \Beta$ and $\sca_{\Beta'} = \sca_{\Beta}$, contradicting the fact that $v\notin \sca_\Beta$, we deduce $\Vert{\Beta'}\Vert > \Vert{\Beta}\Vert$.
\end{proof}

\begin{remark}\label{rem_secord}
We now briefly explain the second-order argument mentioned in the proof above.
Let $Z(x,y)=(f(x,y),g(x,y))$ be a smooth vector field on $\RR^2$,
such that $g(x,0)=0$ for all $x \in \RR$, $f(0,0)=0$, $\frac{\partial g}{\partial y}(0,0)=2c>0$
and $\frac{\partial g}{\partial x}(0,0)=0$. Then by Taylor's formula 
$g(x,y)= 2c y + \eta  C_1  x y +\tilde \eta C_2 y^2$
where $\eta=\eta_{x,y} , \tilde \eta=\tilde \eta_{x,y}\in (0,1)$ 
and $C_1,C_2 \in \RR$. It follows that
for $y \neq 0$ we have $y g(x,y) > cy^2>0$ for all $(x,y) \in B_\epsilon((0,0))$,
$\epsilon > 0$ small enough. This shows that the vector field $Z$ cannot have
an integral curve converging to the origin, unless it is contained in the $x$-axis.
\end{remark}

\section{Properties of critical points of the energy map}\label{sec_critical}

In this section we prove some properties of the
critical points of the energy map $\normmm$. We first show that they are mapped under $\mmm$ onto finitely many $\K$-orbits.

\begin{lemma}\label{lem_finiteness}
The moment map $\mmm$ maps the set of critical point of $\normmm$ onto
a finite number of $\K$-orbits.
\end{lemma}

\begin{proof}
Fix an orthonormal basis $\{e_i\}_{i=1}^N$ for $V$ which diagonalizes the action of $\T$, and let us adopt the notation from Section \ref{sec_torusactions}. If $v = \sum v_i e_i \in V$ then 
\begin{equation}\label{eqn_pialphadiag}
	\alpha \cdot v = \sum \la\alpha, \alpha_i\ra \, v_i e_i,
\end{equation}
for any $\alpha \in \tg$. By \eqref{eqn_defmm} we obtain $\la \mmm(v), \alpha \ra = \tfrac1{\Vert {v}\Vert^2} \cdot \sum \la \alpha, \alpha_i\ra v_i^2$, thus if $\mmm(v) \in \tg$ it follows that 
\begin{equation}\label{eqn_mmcc}
    \mmm(v) = \sum \tfrac{v_i^2}{\Vert {v}\Vert^2} \cdot \alpha_i\,,
\end{equation}
Hence $\mmm(v)\in \Delta_I $, the convex hull of those $\alpha_i$ for which $v_i \neq 0$.

Let now $v_C$ be a critical point of $\normmm$ such that $\Beta := \mmm(v_C)$
is diagonal. Then  $\Beta \cdot v = \Vert \Beta\Vert^2 v$ by Lemma \ref{lem_gradient}. Thus, by equation \eqref{eqn_pialphadiag} 
we must have $\la \Beta, \alpha_i \ra = \Vert \Beta \Vert^2$ for all $i$ such that $v_i \neq 0$. 
Equivalently $\la \Beta, \alpha_i - \Beta \ra = 0$ for all such $i$.
We deduce that $\Beta \perp (\alpha - \Beta)$, for all $\alpha \in \Delta_I$. 
Since $\Beta \in \Delta_I$ by \eqref{eqn_mmcc}, $\Beta$ is the element of minimal norm in $\cca \hca$. Therefore, there are only a finite number of possible $\Beta$'s.
\end{proof}

If $v_C$ is a critical point of the energy map $\normmm$, then the orbit $\K\cdot v_C$ consists of critical points too. By the $\K$-equivariance of the moment map, we may therefore assume that $\Beta:=\mmm(v_C)$ is diagonal. Hence $\Beta^+ \cdot v_C=0$  by Lemma \ref{lem_gradient}, thus $v_C \in \Vnn$. As a consequence
$\Beta^+$ preserves the orthogonal complement of $v_C$ in $V$.

\begin{lemma}\label{lem_Hess}
Let $v_C$ be a critical point of $\normmm$ with $\mmm(v_C) = \Beta$
and let $w \perp v_C$ be an eigenvector of $\Beta^+$ with eigenvalue 
$\lambda_{\Beta^+}\in \RR$ and $\Vert w \Vert = \Vert v_C \Vert$.
Then 
\[
    (\Hess_{v_C}\!\normmm)(w,w) 
    = 4 \cdot \lambda_{\Beta^+} + 2 \cdot \Vert { (d {\mmm})_{v_C}\cdot w} \Vert^2\,.
\]
Moreover, on the subspaces $\kg \cdot  v_C$, $\Vnn$ and
$\big(\kg \cdot v_C+\Vnn \big)^\perp$ the Hessian of $\normmm$ is zero,
non-negative and negative, respectively.
\end{lemma}

\begin{proof}
As in the proof of Lemma \ref{lem_gradient} we set $v(t)=\cos(t)v_C+\sin(t)w$.
This implies for all $t$ that
$	\ddt \normmm(v(t))  = 2 \cdot  \left\la \mmm(v(t)), (d {\mmm})_{v(t)} (v'(t)) \right\ra$ and
differentiating once more yields
\begin{align*}
	\tfrac{{\rm d}^2}{{\rm d}t^2} \big|_0 \normmm(v(t)) 
	&=  \,2 \cdot \ddt\big|_0 \left \la \Beta, (d {\mmm})_{v(t)} \cdot v'(t) \right\ra 
	 + 2 \cdot \Vert  (d {\mmm})_{v_C} \cdot w \Vert^2\,.
\end{align*}
From  (\ref{def_diffmoment}) we deduce
\begin{align*}
 	\ddt\big|_0 \left \la \Beta, (d {\mmm})_{v(t)} \cdot v'(t)  \right\ra 
 	&= \tfrac{2}{\Vert{v_C}\Vert^2} \cdot  \left\la \Beta \cdot w, w\right\ra 
 	 - \tfrac{2}{\Vert{v_C}\Vert^2} \cdot \left\la \Beta \cdot  v_C, v_C \right\ra \,.
\end{align*}
Since $\Beta^+ \cdot v_C = 0$ and $\Beta^+ \cdot w = \lambda_{\Beta^+}\cdot w$, we have that $\Beta \cdot v_C = \Vert{\Beta}\Vert^2 v_C$ and 
$\Beta \cdot  w = \left( \lambda_{\Beta^+} - \Vert{\Beta}\Vert^2 \right)\cdot w$. The formula now follows from $\Vert {w}\Vert = \Vert{v_C}\Vert$.

From the $\K$-invariance of $\normmm$ it follows $\kg \cdot v_C$
lies in the kernel of the Hessian. The nonnegativity of the Hessian
on the subspace $\Vnn$ follows from its definition and
the above formula.
To prove the last assertion observe that for any $w \perp \ggo \cdot  v_C$ one has that $(d\mmm)_{v_C} \cdot w = 0$ by \eqref{def_diffmoment}. 
By Lemma \ref{lem_Vnninv} we get  $\ggo \cdot  v_C \subset  \kg \cdot  v_C + \Vnn$ and the lemma follows.
\end{proof}

Finally, we show that the strata $\sca_\Beta=\G \cdot \Vnn=\K\cdot \Vnnss \subset V\backslash \{0\}$ corresponding to the images of
critical values of the energy map $\normmm$ are smooth, embedded submanifolds.

Let $\G\times_{\Qb} \Vnnss$ be the quotient of $\G \times \Vnnss$ 
with respect to the action of $\Qb$ given by $q \cdot (g,v) = (g q^{-1}, q\cdot v)$.
Since this action is proper  and free,
  $\G \times_{\Qb}  \Vnnss$ is a smooth manifold by a classical result of Koszul (cf.~Proposition 2.2.1 in \cite{Pal61}).
This yields a well-defined, smooth, surjective and $\G$-equivariant map
\[
	\Psi : \G\times_{\Qb} \Vnnss \to \sca_\Beta\,\,;\,\,\, 
	[g, v] \mapsto g \cdot v\,.
\]
Here the action of $\G$ on $\G\times_{\Qb} \Vnnss$ is given by left 
multiplication on the $\G$ factor, 
and it commutes with the action of $\Qb$ mentioned above. Notice also that by linearity of the $\G$-action, the map $\Psi$ is $\RR_{>0}$-equivariant, where $\RR_{>0}$ acts by scalar multiplication on the second factor in $\G \times \Vnnss$ (and the action commutes with the $\Qb$-action, thus passes to the quotient).

\begin{proposition}\label{prop_scasmooth}
Let $v_C \in \sca_\Beta$ be a critical point of $\normmm$ with $\mmm(v_C) = \Beta$.  Then,
the map $\Psi : \G\times_{\Qb} \Vnnss \to \sca_\Beta$ is an embedding. 
\end{proposition}

\begin{proof}
We first prove that $\Psi$ is an immersion at $[e,v_C]$. To that end, it suffices to show that the kernel of the differential of the map $\hat \Psi : \G \times \Vnnss \to \sca_\Beta$, $(g,v) \mapsto g\cdot v$, is given by the tangent space to the $\Qb$-orbit $\Qb \cdot (e,v_C)$ . 
By linearity of the action, the latter is given by 
\[
	T_{(e,v_C)} \, (\Qb \cdot (e,v_C)) = \big\{ (A, - A \cdot v_C) : A \in \qg_\Beta \big\}.
\]
On the other hand, the differential is given by
\[
	{\rm d} \hat \Psi|_{(e,v_C)} (A, w) = A \cdot v_C + w,
\]
and this vanishes if and only if $(A, w) = (A, - A \cdot v_C)$, where $w \in \Vnn$.

Suppose that there exists $A \in \kg$ with $A \cdot v_C \in \Vnn$.	
Then $\exp(t A)\cdot v_C \in \Vnn$ and  $\exp(tA) \in \K$
for all $t \in \RR$. The $\K$-equivariance of the moment map $\mmm$ implies that
$\Vert \mmm(\exp(tA)\cdot v_C)\Vert \equiv \Vert \mmm(v_C)\Vert=\Vert \Beta \Vert$. Since $\exp(tA)\cdot v_C \in \Vnnss \subset  \Vnn$, we deduce
from Lemma \ref{lem_estimatemm} that $\mmm(\exp(tA)\cdot v_C)=\Beta$
for all $t \in \RR$. Thus
$\mmm(\exp(tA)\cdot v_C)= \exp(tA)\cdot \Beta \cdot \exp(-tA)$ for all $t \in \RR$.
Differentiating show $A \in \kg_\Beta \subset \qg_\beta$. This shows the above claim.

In order to show that it is an immersion at any point, we recall that $\Psi$ is $\G$-equivariant and $\RR_{>0}$-equivariant, and apply Corollary \ref{cor_closureGmu}.

In the second step we show that $\Psi$ is injective. By the above, there exists an open neighbourhood 
$\Omega$ of $[e,v_C]$ in $\G \times_{\Qb} \Vnnss$ such that $\Psi|_\Omega$ is injective. 
If $\Psi$ is not globally injective, then there exists $v \in \Vnnss$ and 
$g\in \G \backslash \Qb$ such that $g \cdot v \in \Vnnss$. 
By $\G=\K \Qb$, we may assume that $g = k \in \K \backslash \Kb$. 
Using Corollary \ref{cor_closureGmu}, let $(c_s)\subset \RR_{>0}$ and $(g_s)\subset \G$ be sequences with $g_s = k_s q_s$, $k_s\in \K$, $q_s\in \Qb$, such that $c_s g_s \cdot (k\cdot v) \to \overline{v_C}$, where $\overline{v_C} \in \Vnn$ 
is a critical point of $\normmm$ with $\mmm(\overline{v_C}) = \Beta$.
By Lemma \ref{lem_estimatemm} we deduce $\overline{v_C} \in \Vzero$
and by Lemma \ref{lem_equivWbeta} even $\overline{v_C} \in \Vzeross$.

By compactness of $\K$ we may assume that 
\[
	v_s := c_s q_s \cdot (k\cdot v) \to \overline{v_C}', 
\]
where $\overline{v_C}' \in \K \cdot \overline{v_C}$ is another critical point. Since 
$k \cdot v \in \Vnn$, by Lemma \ref{lem_Vnninv}
$v_s \in \Vnn$ for all $s$,  thus $\overline {v_C}' \in \Vnn$.
 It follows that $\mmm(\overline{v_C}') = \Beta$ 
by Lemma \ref{lem_estimatemm} and $\overline{v_C}' \in \Vzeross$, as above. 
We now write $q_s k = \tilde k_s \tilde q_s$ with $\tilde k_s \in \K$, $\tilde q_s \in \Qb$, and assume that $\tilde k_s \to \tilde k_\infty \in \K$ as \mbox{$s\to \infty$}. By setting $\hat v_C := \tilde k_\infty^{-1} \cdot \overline {v_C}'$, also a critical point, we have that $c_s \tilde q_s \cdot v \to \hat v_C \in \Vnn$  and $\mmm(\hat v_C) = \Beta$ again by Lemma \ref{lem_estimatemm}. In particular, $\tilde k_\infty^{-1} \in \K_\Beta=\Gb \cap \K$ by $\K$-equivariance of $\mmm$ and the very definition of $\Gb$.
Now if we put $\hat k_s := \tilde k_\infty \tilde k_s^{-1} \notin \Kb$, it satisfies $\hat k_s \to \Id$ and $\hat k_s \cdot v_s \to \overline{v_C}'$, or in other words, $[\hat k_s, v_s] \to [e,\overline{v_C}']$. Since $\Psi([\hat k_s, v_s]) = \Psi([e, \hat k_s \cdot v_s])$ and $\hat k_s \notin \Qb$, this contradicts the injectivity of $\Psi$ near 
$[e,\overline{v_C}']=\lim_{s \to \infty}\Psi([e,v_s])$.


In order to show that $\Psi$ is an embedding it remains to show that $\Psi$
is a proper map. So let $([g_i,v_i])_{i\in \NN}$ be a sequence in  $\G \times_{\Qb} \Vnnss$ with
$\Psi([g_i,v_i])\to \Psi([g,v])=g v$ for $i \to \infty$. 
We write $g_i=k_iq_i$
with $k_i \in \K$ and $q_i \in \Qb$. Assuming that $k_i \to k$ for $i\to \infty$
we deduce $q_iv_i \to k^{-1}gv$ for $i \to \infty$.
Thus $[q_i,v_i]=[k_i,q_iv_i]\to [k,k^{-1}gv]$
for $i \to \infty$. This shows the claim.
\end{proof}

An immediate consequence is the following characterization of the parabolic subgroup $\Qb$ (cf.~ Lemma \ref{lem_WbetaQstable}):

\begin{corollary}\label{cor_QbVnnss}
For $v\in \Vnnss$, $g\in \G$ we have that $g\cdot v \in \Vnnss$ if and only if $g\in \Qb$.
\end{corollary}

\begin{proof}
If $g\cdot v \in \Vnnss$ then $[g,v]$, $[e, g\cdot v] \in \G \times_{\Qb} \Vnnss$ and $\Psi([g,v]) = \Psi([e, g\cdot v])$, thus $g\in \Qb$ by injectivity. The converse is clear by Lemma \ref{lem_WbetaQstable}.
\end{proof}

Another application is the fact that the stratum fibers over the compact homogeneous space $\K / \Kb$:

\begin{corollary}\label{cor_Kbfiberbundle}
The map $\Psi_\K : \K \times_{\Kb} \Vnnss \to \sca_\Beta$, $[k,v] \mapsto k\cdot v$, is a diffeomorphism. 
\end{corollary}

\section{Applications}\label{sec_applications}

We collect some applications of the above results. The first is only a restatement
of Lemma \ref{lem_estimatemm} but very important in applications since
it gives non-trivial estimates on strata $\sca_\Beta$ with $\Beta \neq 0$.

\begin{lemma}\label{lem_mainestimate}
For $v \in \Vnn \backslash \{0 \}$ we have that
\[
\Vert{\mmm(v)}\Vert^2 \geq \la \mmm(v), \Beta\ra \geq \Vert{\Beta}\Vert^2
\]
 with equality if and only if $v\in \Vzero$ is a critical point of $\normmm$ with $\mmm(v) = \Beta$.
\end{lemma}

The second application gives information about the isotropy subgroups
\[
  \G_v := \{ g\in \G : g\cdot v = v\}. 
\]

\begin{corollary}\label{cor_autmu}
For any $v \in \Vnnss$ we have $\G_v  \subset \Hb \Ub$ and
$  \G_v \cap \K \subset \Kb$.
\end{corollary}

\begin{proof}
Let $\varphi\in \G_v$. Then $\Psi([\varphi, v]) = v = \Psi([e, v])$.
By the injectivity of $\Psi$ (see Proposition \ref{prop_scasmooth}) we deduce 
$[\varphi,v]=[e,v]$, thus $\varphi \in \Qb=\Gb\Ub$. 
Hence $\varphi = \varphi_\Beta \varphi_u$ with 
 $\varphi_\Beta \in \Gb$ and $\varphi_u \in \Ub$. 
 Let now $w := p_\Beta(v)$, where $p_\Beta : \Vnn \to \Vzero$
 denotes the orthogonal projection. Then
 we have $w \in \Vzeross$, since $\Vnnss=p_{\Beta}^{-1}(\Vzeross)$, and
 $\varphi_\Beta \in \G_w$ by Lemma \ref{lem_pbeta}. 
We write  $\varphi_\Beta = \varphi_{H} \exp(a \Beta)$ with
$\varphi_H \in \Hb$ and $a\in \RR$. 
Recall now that $\Vzeross \subset \Vzero$ by Def.~ \ref{def_Sbeta}
and that  $\exp({-t \beta^+})$ acts trivially  on $\Vzero$
for any $t\in \RR$. Thus,
\[
    (\varphi_H)^n \cdot w = \exp(- n a \Beta) \cdot w =  \left(\exp(na \Vert\beta\Vert^2\Id) \exp(-na \beta^+) \right) \cdot w = e^{n a \Vert \Beta\Vert^2} w,
\]
for all $n\in\ZZ$. It follows that $a=0$, since 
otherwise we would have $0 \in \overline{\Hb \cdot w}$ contradicting  $w \in \Vzeross$ and the very definition of $\Vzeross$.

The last assertion follows from the fact that $\K \cap \Qb  = \Kb$: see
 Lemma \ref{lem_groupsGGbeta}.
\end{proof}




The next application shows that the semi-stable vectors are precisely the stratum corresponding to $\Beta = 0$:

\begin{corollary}\label{cor_semistablesca0}
We have that $V \backslash \nca = \sca_0$. 
\end{corollary}

\begin{proof}
If $v \in V \backslash \nca$, then by Theorem \ref{thm_realGIT}
there exists a minimal vector $ \bar v\in \overline{\G \cdot v}\cap  \sca_0 $, hence
$v \in \sca_0$ by part (ii) of Theorem \ref{thm_stratif}.
To see that $\sca_0 \subset V \backslash \nca$ notice that if $v \in \sca_0 \cap \nca$,
then also the cone $C(\G \cdot v)$ over $\G\cdot v$ is contained in $\nca$, and
by part (iv) of Theorem \ref{thm_realGIT} the same is true for its closure.
Using  that  by Lemma \ref{lem_gradient} 
the gradient of $\normmm$ is tangent to $C(\G \cdot v)$,  we deduce  from part (iii) 
of Theorem \ref{thm_stratif} that for the limit $v_C$ of the negative gradient flow of $\normmm$ starting at $v_0$ we have that $v_C \in \sca_0 \cap \nca$.
Thus, $\mmm(v_C)=0$ and hence $v_C \in V\backslash \nca$,
a contradiction. 
\end{proof}

Finally, our last application of the Stratification Theorem \ref{thm_stratif} generalizes the uniqueness (up to $\K$-action) of zeros of the moment map within the closure of an orbit (Theorem \ref{thm_realGIT}, (i) and (iii)) to critical points of $\normmm$ of higher energy. To that end, one has to restrict to critical points lying in the same stratum, since in general the closure of an orbit may contain several $\K$-orbits of critical points, lying however in strata of higher energy. Let us also mention that this by no means implies that on each stratum there is a unique $\K$-orbit of critical points: these typically come in families of several continuous parameters. 

\begin{corollary}
Assume that $\Id_V \in \ggo$, and let $v\in \sca_\Beta$. Then, there exists 
a critical point $v_C \in \overline{\G \cdot v} \cap \sca_\Beta$ 
of $\normmm$ of unit norm, which is unique up to the action of $\K$.
\end{corollary}

\begin{proof}
Since $\Id_V\in \ggo$, it follows from Lemma \ref{lem_gradient} that the negative gradient flow of $\normmm$ is tangent to $\G$-orbits. Thus, existence of $v_C$ follows by Theorem \ref{thm_stratif}, (iii). 

Regarding uniqueness, assume without loss of generality 
 that $v\in \Vnnss$ (Lemma \ref{lem_WbetaQstable}) and let $v_C, w_C \in \overline{\G\cdot v} \cap \sca_\Beta$ be two critical points of unit norm. Up to the action of $\K$ we may assume $v_C, w_C \in \Vzero$ and
 $\mmm(v_C) = \mmm(w_C) = \Beta$ by Lemma \ref{lem_estimatemm}.
Lemma \ref{lem_equivWbeta} then implies that $v_C,w_C \in \Vzeross$, and from Corollary \ref{cor_QbVnnss} we deduce
\[
    \overline{\G \cdot v} \cap \Vnnss = \overline{\Qb \cdot v} \cap \Vnnss\,.
\]
This implies that $v_C, w_C \in \overline{\Qb \cdot v} \cap \Vzeross$. Using 
$\Gb=  \exp(\RR \Beta) \times \Hb$ and
Lemma \ref{lem_pbeta} we obtain that 
\[
    v_C = p_\Beta(v_C) \in p_\Beta \big(\overline{\Qb \cdot v} \big) \subset \overline{ p_\Beta(\Qb \cdot v)} = \RR_{>0} \cdot \overline{\Hb \cdot p_\Beta(v)},
\]
and the same for $w_C$. The result now follows from Theorem \ref{thm_realGIT}, (i) applied to the $\Hb$-action on $\Vzero$ (see also Lemma \ref{lem_mmZbeta}).
\end{proof}

\begin{appendix}

\section{Real reductive Lie groups}\label{app_reductive}

In this section we collect usefull properties of closed subgroups $\G\subset \Gl(V)$ that satisfy condition (\ref{eqn_Cartandec}), i.e.~ those we call \emph{real reductive Lie groups}.

Firstly, let us mention that there exist in the literature several non-equivalent definitions of this concept. To the best of our knowledge, we can mention at least four: those groups in the Harish-Chandra class \cite[$\S$ 3]{HC75};  Knapp's slightly more general definition \cite[Ch.~ VII, $\S$ 2]{Knapp2e}; Wallach's definition \cite{Wal88}; and Borel's definition \cite[$\S$ 6]{Bor06}. Usually the aim of these is to enlarge the class of real semisimple Lie groups to allow for inductive arguments, since most structural results remain valid for the more robust classes of real reductive groups. Of all of them perhaps the more succinct is that of Borel: a real reductive Lie group is one with finitely many connected components, and whose Lie algebra is reductive (the direct sum of a semisimple Lie subalgebra and its center).

Our interest in linear representations allow us to reduce ourselves to the case of linear groups (i.e.~ those contained in some $\Gl(V)$ for a vector space $V$), thus avoiding lots of technicalities of the general case. Let us also mention that linear groups satisfying Knapp's definition automatically satisfy ours. Conversely, recall that $\Or(n)$ for $n$ even is not real reductive in the sense of Knapp. Regarding a linear group satisfying Borel's definition, it will satisfy \eqref{eqn_Cartandec} provided the center acts on $V$ by semisimple endomorphisms (i.e.~ diagonalizable over $\CC$). 

A first immediate property of a group $\G$ satisfying \eqref{eqn_Cartandec} is that both $\G$ and $\ggo$ are closed under transpose (i.e.~ they are \emph{self-adjoint}, cf.~ \cite{Most55}). 
The following further property, whose proof is based on that given in \cite{Helgason01}, shows that subalgebras in $\pg$ are all conjugate by an element in $\K$:






\begin{proposition}\label{prop_conjugate}
Let $\tg_1, \tg_2 \subset \pg$ be two maximal abelian subalgebras. Then, there exist $k\in \K$ such that $ k \tg_1k^{-1}  = \tg_2$. In particular, for any maximal abelian subalgebra $\tg \subset \pg$ we have that
\[
    \pg = \bigcup_{k\in \K} k\,  \tg \,k^{-1}.
\]
\end{proposition}

\begin{proof}
Recall, that we have by assumption an $\Ad(\K)$-invariant scalar product $\ip$ on $\ggo$. Let $H_i \in \tg_i$ be generic, so that $Z_\pg (H_i) := \{ A \in \pg : [A,H_i] = 0\} = \tg_i$, $i=1,2$. 
Then, there exist $k\in \K$ which minimizes $d:\K \to \RR\,;\,\,k \mapsto \la k H_1 k^{-1}, H_2\ra $, since $\K$ is compact.
At the infinitesimal level this implies that
\[
    \big\la [Z, k H_1 k^{-1}], H_2 \big\ra = 0,
\]
for all $Z \in \kg$. We deduce $ \big\la [H_2, k H_1 k^{-1}], Z\big\ra=0$
for all $Z\in \kg$,
using \eqref{eqn_assumipggo} and that $H_i \in \pg \subset \Sym(V)$, $Z\in \kg \subset \sog(V)$. Since $[\pg,\pg] \subset \kg$, we have $[H_2, k H_1 k^{-1}] \in \kg$, thus $[H_2, k H_1 k^{-1}] = 0$. By definition of $H_2$ this yields $k H_1 k^{-1} \in \tg_2$. For any $A_2 \in \tg_2$ we deduce $[k^{-1} A_2 k, H_1] = 0$,
from which $A_2 \in \tg_1$ by definition of $H_1$. Therefore, $k^{-1} \tg_2 k \subset \tg_1$, which also reads as $\tg_2 \subset k \tg_1 k^{-1}$. By maximality of $\tg_2$ this implies that $\tg_2 = k \tg_1 k^{-1}$. 

The last assertion follows from the previous one, by choosing any $A\in \pg$ and extending $\RR A$ to a maximal abelian subalgebra.
\end{proof}

\begin{corollary}\label{cor_KTK}
For any maximal abelian subalgebra $\tg \subset \pg$ we have that $\G = \K \T \K$, where $\T:= \exp(\tg)$.
\end{corollary}

\begin{proof}
We have $\G = \K \,\exp(\pg)$ and Proposition \ref{prop_conjugate} imples that $\exp(\pg) \subset \K \T \K$.
\end{proof}

Let us now fix a maximal abelian subalgebra $\tg \subset \pg$. Choose an orthonormal basis for $(V,\ip)$ such that $\tg$ is contained in $\dg$, the set of diagonal matrices in $\glg(V)$. Moreover, by maximality of $\tg$ it also holds that $\tg = \ggo \cap \dg$. Denote by $\ug$ (resp.~ $\ug^t$) the nilpotent subalgebra of $\glg(V)$ of strictly lower (resp.~ upper) triangular matrices, and by $\mathsf{U} := \exp(\ug)$ the corresponding analytic subgroup of $\Gl(V)$.

We now look at the root space decomposition of $\glg(V)$ with respect to $\tg$. More precisely, $\ad(\tg)$ is a commuting family of symmetric endomorphisms of $\glg(V)$, hence there exists a finite subset $\Sigma \subset \tg^*$ and a decomposition into common eigenspaces
\[
    \glg(V) = \bigoplus_{\lambda \in \Sigma} \glg(V)_\lambda,
\]
where for $E\in \tg$ we have $\ad(E)|_{\glg(V)_\lambda} = \lambda(E) \cdot \Id_{\glg(V)_\lambda}$. Since for any $E\in \tg$ we have that $\ad(E)|_\ug = -\ad(E)|_{\ug^t}$, as is well-known $\Sigma$ can be decomposed as a disjoint union $\Sigma = \{ 0\} \cup \Sigma^+ \cup \Sigma^-$, such that
\[
    \tg \subset \glg(V)_0, \qquad \bigoplus_{\lambda \in \Sigma^+} \glg(V)_\lambda \subset \ug, \qquad \bigoplus_{\lambda \in \Sigma^-} \glg(V)_\lambda \subset \ug^t.
\]

Consider now the corresponding subalgebras of $\ggo$, $\ngo := \ggo \cap \ug$, $\ngo^t := \ggo \cap \ug^t$, and the analytic subgroups $\T := \exp(\tg)$, $\mathsf{N} := \exp(\ngo) = (\G \cap \mathsf{U})_0$, corresponding to $\tg, \ngo$, respectively. Since $\ad(\tg)$ preserves $\ggo$, we have the induced root space decomposition
\[
    \ggo = \bigoplus_{\lambda \in \Sigma} \ggo_\lambda,
\]
where now some $\ggo_\lambda$ might be trivial. The above discussion implies that 
\[
 \ngo = \bigoplus_{\lambda \in \Sigma^+} \ggo_\lambda\,\,\,,
 \quad \ngo^t = \bigoplus_{\lambda \in \Sigma^-} \ggo_\lambda
 \quad \textrm{and}\quad \tg = \ggo_0\,, 
\] 
thus $\ggo = \tg \oplus \ngo \oplus \ngo^t$.

\begin{proposition}\label{prop_Iwasawa}
If $\G\subset \Gl(V)$ is closed and satisfies  \eqref{eqn_Cartandec}, then $\G = \K \T \mathsf{N}$.
\end{proposition}

\begin{proof}
Assume first that $\G$ is connected. Since $\tg \oplus \ngo$ is a subalgebra of $\ggo$, it follows that $\T \mathsf{N}$ is a closed connected subgroup of $\G$, intersecting $\K$ trivially: the only lower triangular orthogonal matrix with positive eigenvalues is the identity. The result in this case would follow provided we show that $\ggo = \kg \oplus \tg \oplus \ngo$. To that end, let $E\in \ggo = \bigoplus_{\lambda \in \Sigma} \ggo_\lambda$, say $E\in \ggo_\lambda$. If $\lambda = 0$ then $E\in \tg$. If $\lambda \in \Sigma^+$ then $E\in \ngo$. And finally, for $\lambda \in \Sigma^-$, we have $E^t \in \ngo$ thus 
\[
    E = (E-E^t) + E^t \in \kg \oplus \ngo.
\]
The general case follows from the previous one: we know that $\G_0 = \K_0 \T \mathsf{N}$. On the other hand, the connected component of the identity $\G_0$ is given by $\G_0 = \K_0 \exp(\pg)$ thanks to  \eqref{eqn_Cartandec}, thus $\G / \G_0 \simeq \K / \K_0$ and hence $\G = \K \T \mathsf{N}$.
\end{proof}

\section{The parabolic subgroup $\Qb$}\label{app_groups}

Let $\G \subset \Gl(V)$ be a closed subgroup satisfying  \eqref{eqn_Cartandec}
with Lie algebra $\ggo \subset \glg(V)$ and Cartan decomposition $\ggo = \kg \oplus \pg$, $\kg = \ggo \cap \sog(V)$, $\pg = \ggo \cap \Sym(V)$. 
We will describe in this section some important subgroups of $\G$ associated with a fixed element $\Beta \in \pg$. 

Consider for such a fixed $\Beta \in \pg$ the adjoint map 
\[
    \ad(\Beta) : \ggo \to \ggo, \qquad A \mapsto [\Beta, A].
\] 
Assume that $\ad(\Beta) : \ggo \to \ggo$ is a symmetric endomorphism with respect to some scalar product on $\ggo$: see \eqref{eqn_assumipggo}.
If $\ggo_r$ is the eigenspace of $\ad(\Beta)$ with eigenvalue $r\in \RR$ then $\ggo = \bigoplus_{r\in \RR} \ggo_r$, and we set
\[
    \ggo_\Beta := \ggo_0 = \ker (\ad(\Beta) ), \qquad \ug_\Beta := \bigoplus_{r> 0} \ggo_r, \qquad \qg_\Beta := \ggo_\Beta \oplus \ug_\Beta.
\]

\begin{definition}\label{def_groupsapp}
We denote by 
\[
    \Gb := \{ g \in \G : g \Beta g^{-1} = \Beta \}, \quad \Ub := \exp(\ug_\Beta)\quad \textrm{and}
     \quad \Qb := \Gb \Ub
\] 
the centralizer of $\Beta$ in $\G$, the unipotent subgroup associated with $\Beta$, and the parabolic subgroup associated with $\Beta$, respectively. 
\end{definition}

To describe these groups more explicitly, let us decompose $V = V_1 \oplus \cdots  \oplus V_m$ as a sum of $\Beta$-eigenspaces corresponding to the real eigenvalues $\lambda_1 < \cdots < \lambda_m$
with multiplicities $n_1,\ldots, n_m \in \NN$.
For simplicity we suppose that $m=3$, the general case being completely analogous.
 With respect to a suitable orthonormal basis of $V$ the map $\beta = (\lambda_1 \Id_1,\lambda_2\Id_2,\lambda_3 \Id_3)$ is diagonal, and  we have 
\begin{eqnarray*}
    \Gb &=& \G \cap \left\{ 
              \left( \begin{array}{ccc}
            g_1 & 0 & 0 \\ 0 & g_2 &  0\\ 0 & 0 &  g_3
            \end{array}\right)  : g_i \in \Gl(V_i)
        \right\},   \\
    \Ub &=& \G \cap \left\{ 
              \left( \begin{array}{ccc}
            \Id_1 & 0 & 0 \\ g_{21} & \Id_2 &  0\\ g_{31} & g_{32} &  \Id_3
            \end{array}\right) : g_{ij} \in \End(V_j,V_i), \,\, i>j
        \right\},   \\
    \Qb &=& \G \cap \left\{ 
              \left( \begin{array}{ccc}
            g_{11} & 0 & 0 \\ g_{21} & g_{22} &  0\\ g_{31} & g_{32} &  g_{33}
            \end{array}\right)  : g_{ij} \in \End(V_j,V_i) \mbox{ if } i>j,  \,\,  g_{ii} \in \Gl(V_i)
        \right\}.
\end{eqnarray*}

\begin{lemma}\label{lem_groupsGGbeta}
The groups $\Gb, \Ub$, $\Qb$ are closed in $\G$, and their Lie algebras are given respectively by $\ggo_\Beta$, $\ug_\Beta$, $\qg_\Beta$.  We have that
\[
    \G = \K \Qb, \qquad \K \cap \Qb = \K \cap \Gb =: \Kb.
\] 
Moreover, $\Ub$ is connected and normal in $\Qb$, $\Gb$ is reductive and satisfies \eqref{eqn_Cartandec}, and we have $\Ub \cap \Gb = \{ \Id_V \}$. 
\end{lemma}

\begin{proof}
A simple computation shows that $[\ggo_r, \ggo_s] \subset \ggo_{r+s}$, thus $\ug_\Beta$ is a Lie subalgebra of $\ggo$ and an ideal in $\qg_\Beta$. Hence $\Ub$ is a closed subgroup of $\G$, which is normal in $\Qb$. The claims for $\Gb$ are well-known.

The fact that $\G = \K \Qb$ follows at once from Proposition \ref{prop_Iwasawa}, applied to a maximal abelian subalgebra $\tg \subset \pg$ containing $\Beta$: in this case, we have that $\T \mathsf{N} \subset \Qb$.

If $g\in \Gb \cap \Ub$, say $g = \exp(N)$ with $N\in \ug_\Beta$, then $g$ preserves the eigenspaces of $\Beta$, thus so does $N$. Hence $N\in \ggo_\Beta$, therefore $N = 0$ and $g = \Id_V$. 

Now observe that $\Gb^t = \Gb$, so in particular we also have $\Gb \cap \Ub^t = \{ \Id_V\}$. Also, since $\ggo_r^t = \ggo_{-r}$, we obtain in the same way as above that $\Ub^t \cap \Qb = \{ \Id_V \}$.  Since $\Ub$ is normal in $\Qb$ we may write $\Qb = \Gb \Ub = \Ub \Gb$, and in particular $\Qb^t = \Gb \Ub^t$. From this observations it follows that $\Qb \cap \Qb^t = \Gb$, and it is now clear that $\K \cap \Qb = \Kb$.

Finally, if we define 
 $\pg_\Beta = \ggo_\Beta \cap \Sym(V) = \ggo_\Beta \cap \pg$, from the fact that $\Gb^t = \Gb$ we clearly have that $\Gb = \Kb \exp(\pg_\Beta)$.
\end{proof}

The following characterization turns out to be very useful in the applications.

\begin{lemma}\label{lem_groupslimit}
For $\Beta \in \pg$ we have that
\begin{eqnarray*}
        \Gb &=& \{ g\in \G : g \exp(\Beta) g^{-1} = \exp(\Beta) \}, \\
        \Ub &=& \big\{ g\in \G : \lim_{t\to\infty} \exp(-t \beta) g \exp(t\beta) = \Id \big\},    \\
        \Qb &=& \big\{ g\in \G : \lim_{t\to\infty} \exp(-t \beta) g \exp(t\beta) \hbox{ exists} \big\}.
\end{eqnarray*}
\end{lemma}
\begin{proof} 
Notice that $g\in \Gb$ if and only if $g$ preserves all the eigenspaces of $\Beta$. Since these coincide with the eigenspaces of $\exp(\Beta)$, the first assertion follows. For the other two claims, as above we work in the case where $\Beta$ has only $3$ different eigenvalues, bearing in mind that the general case follows in exactly the same way. Adopting the same notation as in the paragraph following Definition \ref{def_groupsapp}, for $g=(g_{ij})_{1\leq i,j \leq 3} \in \Gl(V)$ we have that the limit
\begin{eqnarray*}
\lefteqn{  \left( \begin{array}{ccc}
    e^{-t\lambda_1}\!\! && \\ &\!\! e^{-t\lambda_2}\!\! & \\ &&\!\! e^{-t\lambda_3}
 \end{array}\right)
 \cdot 
  \left( \begin{array}{ccc}
    g_{11} &g_{12 }& g_{13 }\\ g_{21 }& g_{22 }& g_{23} \\ g_{31}& g_{32 }& g_{33}
 \end{array}\right)
 \cdot 
  \left( \begin{array}{ccc}
    e^{t\lambda_1}\!\! && \\ &\!\! e^{t\lambda_2} \!\!& \\ &&\!\! e^{t\lambda_3}
 \end{array}\right) } &&\\
 &&
  =\left( \begin{array}{ccc}
    g_{11} &e^{-t(\lambda_1-\lambda_2)}g_{12 }&e^{-t(\lambda_1-\lambda_3)} g_{13 } \\
     e^{-t(\lambda_2-\lambda_1)}g_{21 }& g_{22 }&e^{-t(\lambda_2-\lambda_3)} g_{23}  \\ 
     e^{-t(\lambda_3-\lambda_1)}g_{31}& e^{-t(\lambda_2-\lambda_3)}g_{32 }& g_{33}
 \end{array}\right) \longrightarrow
 \left( \begin{array}{ccc}
    g_{11} &&\\
     & g_{22 }& \\ 
    && g_{33}
 \end{array}\right)
\end{eqnarray*}
as $t \to \infty$ exists if and only if $g_{12}=g_{13}=g_{13}=0$, and it is $\Id_V$ if and only if in addition we have $g_{ii} = \Id_{V_i}$ for $i=1,2,3$.
\end{proof}

\end{appendix}

\bibliography{../bib/ramlaf2}
\bibliographystyle{amsalpha}

\end{document}